\documentclass{amsart}
\usepackage{amssymb,amsmath}
\usepackage{graphicx}
\usepackage{color}
\theoremstyle{plain}
\newtheorem{thm}{Theorem}[section]

\newtheorem{cor}[thm]{Corollary}
\newtheorem{lem}[thm]{Lemma}
\newtheorem{prop}[thm]{Proposition}

\theoremstyle{definition}
\newtheorem{defn}{Definition}[section]

\theoremstyle{remark}
\newtheorem{rem}{Remark}[section]
\newtheorem{ex}{Example}[section]

\numberwithin{equation}{section}

\newcommand{\ra}{\rightarrow}
\newcommand{\Ra}{\Rightarrow}

\newcommand{\Lra}{\Leftrightarrow}

\begin{document}

\title{The $L^2$-cutoffs for reversible Markov chains}

\author[G.-Y. Chen]{Guan-Yu Chen$^1$}

\author[J.-M. Hsu]{Jui-Ming Hsu$^2$}


\author[Y.-C. Sheu]{Yuan-Chung Sheu$^3$}


\address{$^1$Department of Applied Mathematics, National Chiao Tung University, Hsinchu 300, Taiwan}
\email{gychen@math.nctu.edu.tw}

\address{$^2$Department of Applied Mathematics, National Chiao Tung University, Hsinchu 300, Taiwan}
\email{raymand.hsu@msa.hinet.net}


\address{$^3$Department of Applied Mathematics, National Chiao Tung University, Hsinchu 300, Taiwan}
\email{sheu@math.nctu.edu.tw}



\begin{abstract}
In this article, we considers reversible Markov chains of which $L^2$-distances can be expressed in terms of Laplace transforms. The cutoff of Laplace transforms was first discussed by Chen and Saloff-Coste in \cite{CSal10}, while we provide here a completely different pathway to analyze the $L^2$-distance. Consequently, we obtain several considerably simplified criteria and this allows us to proceed advanced theoretical studies, including the comparison of cutoffs between discrete time lazy chains and continuous time chains. For an illustration, we consider product chains, a rather complicated model which could be involved to analyze using the method in \cite{CSal10}, and derive the equivalence of their $L^2$-cutoffs.

\end{abstract}

\maketitle

\section{Introduction}

Let $\mathcal{S}$ be a finite set, $K$ be a stochastic matrix indexed by $\mathcal{S}$ and $\pi$ be a probability on $\mathcal{S}$. We write the triple $(\mathcal{S},K,\pi)$ for an irreducible discrete time Markov chain on $\mathcal{S}$ with transition matrix $K$ and stationary distribution $\pi$. Concerning the continuous time case, we write $(\mathcal{S},L,\pi)$ for an irreducible continuous time Markov chain on $\mathcal{S}$ with infinitesimal generator $L$ and stationary distribution $\pi$. By setting $H_t=e^{tL}$, it is well-known that $H_t(x,\cdot)$ converges to $\pi$ for all $x\in\mathcal{S}$. If $K$ is aperiodic, then $K^n(x,\cdot)$ converges to $\pi$ for all $x\in\mathcal{S}$.

To study the convergence of Markov chains, we introduce the $L^2$-distance as follows. For irreducible Markov chains $(\mathcal{S},K,\pi)$ and $(\mathcal{S},L,\pi)$ with initial distribution $\mu$, we briefly write them as $(\mu,\mathcal{S},K,\pi)$ and $(\mu,\mathcal{S},L,\pi)$ and define their $L^2$-distances respectively by
\[
 d_2(\mu,m)=\|\mu K^m/\pi-1\|_{L_2(\pi)}=\left(\sum_{y\in \mathcal{S}}\left|\frac{\mu K^m(y)}{\pi(y)}-1\right|^2\pi(y)\right)^{1/2}
\]
and
\[
 d_2(\mu,t)=\|\mu H_t/\pi-1\|_{L_2(\pi)}=\left(\sum_{y\in \mathcal{S}}\left|\frac{\mu H_t(y)}{\pi(y)}-1\right|^2\pi(y)\right)^{1/2}.
\]
Accordingly, the $L^2$-mixing time is defined by
\[
 T_2(\mu,\epsilon)=\min\{t\ge 0|d_2(\mu,t)\le\epsilon\},
\]
where $t$ refers to non-negative integers for discrete time chains and to non-negative reals for continuous time chains.

For a reversible transition matrix $K$ with eigenvalues $\beta_0=1,\beta_1,...,\beta_{|\mathcal{S}|-1}$ and $L^2(\pi)$-orthonormal right eigenvectors $\phi_0=\mathbf{1},...,\phi_{|\mathcal{S}|-1}$, the $L^2$-distance can be expressed as
\begin{equation}\label{eq-l2dis}
 d_2(\mu,m)^2=\sum_{i>0}|\mu(\phi_i)|^2\beta_i^{2m},
\end{equation}
where $\mathbf{1}$ denotes the constant function with value $1$ and  $\mu(\phi_i):=\sum_x\mu(x)\phi_i(x)$. Similarly, in the continuous time case, if $L$ is reversible with eigenvalues $-\lambda_0=0,-\lambda_1,...,-\lambda_{|\mathcal{S}|-1}$ and $L^2(\pi)$-orthonormal right eigenvectors $\phi_0=\mathbf{1},...,\phi_{|\mathcal{S}|-1}$, then the $L^2$-distances can be expressed as
\begin{equation}\label{eq-l2cts}
 d_2(\mu,t)^2=\sum_{i>0}|\mu(\phi_i)|^2e^{-2t\lambda_i}.
\end{equation}
A proof of (\ref{eq-l2dis}) and (\ref{eq-l2cts}) is available in \cite{LPW08,SC97}. Note that, for continuous time chains, the $L^2$-distance in (\ref{eq-l2cts}) can be identified with a Lebesgue-Stieltjes integral in the way that
\begin{equation}\label{eq-laplace}
 d_2(\mu,t)^2=\int_{(0,\infty)}e^{-t\lambda}dV(\lambda),\quad\forall t\ge 0,
\end{equation}
where $V$ is a nondecreasing function defined by
\begin{equation}\label{eq-V}
 V(\lambda)=\sum_{i=1}^{j-1}|\mu(\phi_i)|^2,\quad\text{if}\quad 2\lambda_{j-1}\le \lambda<2\lambda_j,\,\,1\le j\le |\mathcal{S}|,
\end{equation}
with the convention $\sum_{i=1}^0:=0$ and $\lambda_{|\mathcal{S}|}:=\infty$ and $\lambda_i$'s are arranged in a nondecreasing order. In the same spirit, the $L^2$-distance of discrete time chains in (\ref{eq-l2dis}) can be also written in the form of (\ref{eq-laplace}) with non-negative integer $t$ when $\beta_i$'s are rearranged in the order of $|\beta_i|\ge|\beta_{i+1}|$ and, in (\ref{eq-V}), $\lambda_i$ is replaced by $-\log|\beta_i|$ along with the convention $-\log 0:=\infty$ and $-\log|\beta_{|\mathcal{S}|}|:=\infty$. In fact, for the discrete time case, the definition of $V$ in (\ref{eq-V}) is only valid for $0\le j\le j_0+1$, where $j_0$ is the largest $j$ such that $|\beta_j|>0$. It is worthwhile to remark that, for reversible Markov processes with initial distribution $\mu$ and stationary distribution $\pi$, the $L^2$-distance is still of the form in (\ref{eq-laplace}) when the density $d\mu/d\pi$ has a finite $L^2(\pi)$-norm. See Section 4 of \cite{CSal10} for more details in this aspect. Throughout this article, we focus on reversible Markov chains with finite states, while most results are valid in a more general setting.

The cutoff phenomenon was introduced by Aldous and Diaconis in 1980's, see e.g. \cite{A83,AD86,AD87,DS81,DS87}, for the purpose of capturing a phase transit arouse in the evolution of Markov chains. To see a definition of cutoffs in the $L^2$-distance, consider a family of irreducible discrete time Markov chains $\mathcal{F}=(\mu_n,\mathcal{S}_n,K_n,\pi_n)_{n=1}^\infty$. For $n\ge 1$, let $d_{n,2}$ be the $L^2$-distance of the $n$th chains in $\mathcal{F}$ and $T_{n,2}$ be the corresponding $L^2$-mixing time. The family $\mathcal{F}$ is said to present a $L^2$-cutoff if there is a sequence $(t_n)_{n=1}^\infty$ such that
\begin{equation}\label{eq-l2cutoffdef}
 \lim_{n\ra\infty}d_{n,2}(\mu_n,\lceil(1+a)t_n\rceil)=0,\quad
 \lim_{n\ra\infty}d_{n,2}(\mu_n,\lfloor(1-a)t_n\rfloor)=\infty,
\end{equation}
for all $a\in(0,1)$, where $\lceil u\rceil:=\min\{z\in\mathbb{Z}|z\ge u\}$ and $\lfloor u\rfloor:=\max\{z\in\mathbb{Z}|z\le u\}$. In the continuous time case, the $L^2$-cutoff is defined in the same way except the removal of $\lceil\cdot\rceil,\lfloor\cdot\rfloor$ and, in either case, the sequence $(t_n)_{n=1}^\infty$, or briefly $t_n$, is called a $L^2$-cutoff time. It has been developed in \cite{CSal08} that the cutoff is closely related to the mixing time and the result says that, in the discrete time case, if $T_{n,2}(\mu_n,\epsilon_0)\ra\infty$ for some $\epsilon_0>0$, then $\mathcal{F}$ has a $L^2$-cutoff if and only if
\begin{equation}\label{eq-cutmix}
 \lim_{n\ra\infty}\frac{T_{n,2}(\mu_n,\epsilon)}{T_{n,2}(\mu_n,\delta)}=1,\quad\forall \epsilon,\delta\in(0,\infty)
\end{equation}
For the exception that a $L^2$-cutoff appears with bounded $L^2$-mixing time, the $L^2$-distance would drop from infinity to zero within one or two steps. As the time is integer-valued, the limit in (\ref{eq-cutmix}) could fail in this instance. For the continuous time case, the $L^2$-cutoff is also equivalent to (\ref{eq-cutmix}) without the assumption of $T_{n,2}(\mu_n,\epsilon_0)\ra\infty$. As $d_{n,2}(\mu_n,\cdot)$ is non-increasing, one can see from (\ref{eq-l2cutoffdef}) that $T_{n,2}(\mu_n,\epsilon)$ is an eligible $L^2$-cutoff time.

In an ARCC workshop in 2004, Peres proposed a heuristic idea to examine the existence of cutoffs, which said
\begin{equation}\label{eq-conj}
 \text{Cutoff exists}\quad\Lra\quad\text{Mixing time $\times$ Spectral gap $\ra\infty$},
\end{equation}
where the spectral gap refers to the smallest non-zero eigenvalue of $-L$ in the continuous time case and to the logarithm of the reciprocal of the second largest singular value of $K$ in the discrete time case. Such a criterion has been proved to work on a large class of Markov chains but, unfortunately, it could fail in general. In \cite{DS06}, Disconis and Saloff-Coste proved this conjecture for birth and death chains in separation. In \cite{CSal08}, Chen and Saloff-Coste declared the accuracy of (\ref{eq-conj}) for reversible chains in the maximal $L^p$-distance. In \cite{BSU14}, Basu et. al. clarified (\ref{eq-conj}) for lazy random walks on trees in the maximal total variation. In \cite{CSal10}, Chen and Saloff-Coste considered reversible chains with specified initial distributions and produced a criterion similar to (\ref{eq-conj}) to identify the $L^2$-cutoff. However, counterexamples to (\ref{eq-conj}) were respectively observed by Aldous and Pak, and we refer the readers to \cite[Section 6]{CSal08} and \cite[Chapter 18]{LPW08} for illustrations of their ideas.

The object of this article is to provide a viewpoint somewhat different from what was introduced in \cite{CSal10} so that further developments, say comparisons of cutoffs, can work and rather complicated models, say product chains, can be analyzed. In the following, we illustrates one of the main results in this article.

\begin{thm}\label{t-main}
Let $\mathcal{F}=(\mu_n,\mathcal{S}_n,L_n,\pi_n)_{n=1}^\infty$ be a family of irreducible and reversible continuous time finite Markov chains. For $n\ge 1$, let $\lambda_{n,0}=0<\lambda_{n,1}\le\cdots\le\lambda_{n,|\mathcal{S}_n|-1}$ be eigenvalues of $-L_n$ with $L^2(\pi_n)$-orthonormal right eigenvectors $\phi_{n,0}=\mathbf{1},...,\phi_{n,|\mathcal{S}_n|-1}$. For $c>0$, set
\begin{equation}\label{eq-jn}
 j_n(c):=\min\left\{j\ge 1\Bigg|\sum_{i=1}^j|\mu_n(\phi_{n,i})|^2>c\right\}
\end{equation}
and
\begin{equation}\label{eq-taun}
 \tau_n(c):=\max_{j\ge j_n(c)}\left\{\frac{\log\left(1+\sum_{i=1}^j|\mu_n(\phi_{n,i})|^2\right)}{2\lambda_{n,j}}\right\}.
\end{equation}
Suppose that $\pi_n(|\mu_n/\pi_n|^2)\ra\infty$. Then, the following are equivalent.
\begin{itemize}
\item[(1)] $\mathcal{F}$ has a $L^2$-cutoff.

\item[(2)] There is $\delta>0$ such that
\[
 \lim_{n\ra\infty}T_{n,2}(\mu_n,\delta)\lambda_{n,j_n(c)}=\infty,\quad\forall c>0.
\]

\item[(3)] For all $c>0$,
\[
 \lim_{n\ra\infty}\tau_n(c)\lambda_{n,j_n(c)}=\infty.
\]
\end{itemize}
Moreover, if \textnormal{(2)} holds, then
\[
 |T_{n,2}(\mu_n,\delta)-T_{n,2}(\mu_n,\epsilon)|
 =O\left(1/\lambda_{n,j_n(c)}\right),\quad\forall \epsilon,\delta,c\in(0,\infty),
\]
where two sequences of positive reals, $a_n$ and $b_n$, satisfy $a_n=O(b_n)$ if $a_n/b_n$ is bounded. If \textnormal{(3)} holds, then
\[
 |T_{n,2}(\mu_n,\epsilon)-\tau_n(c)|=O\left(\sqrt{\tau_n(c)/\lambda_{n,j_n(c)}}\right),\quad\forall \epsilon,c\in(0,\infty).
\]
\end{thm}

Concerning Theorem \ref{t-main}(2), as $\lambda_{n,j_n(c)}$ is non-decreasing in $c$, it suffices to focus on the limit with small enough positive $c$. Such an observation is also applicable to Theorem \ref{t-main}(3) but the reasoning is not obvious to see since $\tau_n(c)$ is non-increasing in $c$. The reader is referred to Lemma \ref{l-cond34} for details of the above discussions. On the other hand, it's worthwhile to note that $\lambda_{n,j_n(c)}$ is not necessarily the spectral gap in (\ref{eq-conj}). Following this fact, one may create a counterexample to (\ref{eq-conj}) in the way that a $L^2$-cutoff exists but the product in (\ref{eq-conj}) is bounded. For advanced profiles of cutoffs, the last two bounds in Theorem \ref{t-main} say that if a $L^2$-mixing time is selected as the $L^2$-cutoff time, then the cutoff window is at most $1/\lambda_{n,j_n(c)}$; if $\tau_n(c)$ is designated as the $L^2$-cutoff time, then the cutoff window is at most $\sqrt{\tau_n(c)/\lambda_{n,j_n(c)}}$, which is of order bigger than $1/\lambda_{n,j_n(c)}$. We refer the reader to \cite{CSal08,CSal10} for more discussions on cutoff windows.

Compared with Theorems 5.1 and 5.3 in \cite{CSal10}, Theorem \ref{t-main} looks more familiar to (\ref{eq-conj}), though the spectral gap is updated to a modified version. In addition to the right side of (\ref{eq-conj}), there is in fact an auxiliary condition for the $L^2$-cutoff in \cite{CSal10} and this makes it difficult to do any further theoretical development. The tradeoff of removing the side condition in \cite{CSal10} is to strengthen the requirement in (\ref{eq-conj}) up to the extent of Theorem \ref{t-main}, but the benefit from the simplification of cutoff criteria leads to comparisons between discrete time lazy chains and continuous time chains as shown in Theorem \ref{t-comparison} and Corollary \ref{c-comparison}. Naively, one may expect to refine Theorem \ref{t-main} so that, for some $c>0$, the limits in conditions (2) and (3) are sufficient for an $L^2$-cutoff. However, there are indeed counterexamples against this conjecture and we demonstrate one in Example \ref{ex-counterexample}.

For another application of the general results, we consider products of Markov chains (briefly, product chains) in Section \ref{s-cexp}. Concerning product chains, the hitting time and spectral information are discussed in \cite{AF,LPW08,SC97} and a detailed analysis on the mixing time is made in \cite{BLY06}. In this article, we introduce Proposition \ref{p-prod} to reduce the complexity of spectral information and provide in Theorem \ref{t-prod} a much simplified criterion on the judgement of $L^2$-cutoff. Particularly, we study products of two-state chains in a rather concrete setting and gather the results in Theorems \ref{t-prod2state}-\ref{t-prod2state2}.

To see a practical issue related to product chains, let's consider a machinery with a large number of components. Each component has two states and evolves independently in the way that, given the state is renewed, an exponential clock is activated and the component changes to the other state when the clock rings. Concerning the effect of some external force, we assume that each component could speed up or slow down its evolution but still operates independently. The question here is how (the existence of cutoffs) and when (the mixing time) this machinery gets close to its stability. For convenience, we quantize this problem as follows. For $n\ge 1$, let
\begin{equation}\label{eq-machinery1}
 \mathcal{X}_n=\{0,1\},\quad M_n=\left(\begin{array}{cc}-A_n&A_n\\B_n&-B_n\end{array}\right),\quad p_n>0,
\end{equation}
which denote respectively the state space, the infinitesimal generator and the accelerating constant of the $n$th component. Concerning the irreducibility of chains, we assume $A_n,B_n\in(0,1)$ and, obviously, $\nu_n=(B_n,A_n)/(A_n+B_n)$ is the stationary distribution of $M_n$. Let $x_n,\ell_n$ be positive integers and set
\begin{equation}\label{eq-machinery2}
 L_n=q_n^{-1}\sum_{i=x_n}^{x_n+\ell_n-1}p_iI_{x_n}\otimes\cdots\otimes I_{i-1}\otimes M_i\otimes I_{i+1}\otimes\cdots\otimes I_{x_n+\ell_n-1},
\end{equation}
where $q_n=p_{x_n}+\cdots+p_{x_n+\ell_n-1}$, $I_j$'s are $2$-by-$2$ identity matrices and $M\otimes M'$ denotes the tensor product of matrices $M$ and $M'$. Clearly, $\pi_n=\nu_{x_n}\times\cdots\times\nu_{x_n+\ell_n-1}$ is the stationary distribution of $L_n$.

\begin{thm}\label{t-machinery}
Referring to \textnormal{(\ref{eq-machinery1})-(\ref{eq-machinery2})}, let $\mathcal{G}=(\delta_{\mathbf{0}},\mathcal{S}_n,L_n,\pi_n)_{n=1}^\infty$, where $\mathcal{S}_n=\{0,1\}^{\ell_n}$ and $\delta_{\mathbf{0}}$ is the Dirac delta function on the zero vector. Suppose that
\[
 A_n+B_n=A_1+B_1,\quad\forall n\ge 1,\quad 0<\inf_{n\ge 1}\frac{A_n}{B_n}\le\sup_{n\ge 1}\frac{A_n}{B_n}<\infty.
\]
\begin{itemize}
\item[(1)] If $p_i=e^{ai}$ with $a>0$, then $\mathcal{G}$ has no $L^2$-cutoff.

\item[(2)] If $p_i=\exp\{a[\log(1+i)]^b\}$ with $a>0$ and $b>0$, then
\[
 \mathcal{G}\text{ has a $L^2$-cutoff}\quad\Lra\quad \min\{x_n,\ell_n\}\ra\infty.
\]
Further, if $\mathcal{G}$ has a $L^2$-cutoff, then
\begin{equation}\label{eq-Gcutofftime}
 T_{n,2}(\delta_{\mathbf{0}},\epsilon)=\frac{\kappa_n}{2(A_1+B_1)p_{x_n}}
 +O\left(\frac{\sqrt{\kappa_n}}{(A_1+B_1)p_{x_n}}\right),\quad\forall\epsilon>0,
\end{equation}
where $\kappa_n=\min\{(\log x_n-b\log\log x_n),\log\ell_n\}$.

\item[(3)] If $p_i=[\log(1+i)]^a$ with $a>0$, then
\[
 \mathcal{G}\text{ has a $L^2$-cutoff}\quad\Lra\quad\begin{cases}\min\{x_n,\ell_n\}\ra\infty&\text{for }a\ge 1,\\\ell_n\ra\infty&\text{for }0<a<1.\end{cases}
\]
Further, if $a\ge 1$ and $\min\{x_n,\ell_n\}\ra\infty$, then \textnormal{(\ref{eq-Gcutofftime})} holds with $\kappa_n=\min\{(\log x_n),(\log \ell_n)\}$. If $0<a<1$ and $\ell_n\ra\infty$, then \textnormal{(\ref{eq-Gcutofftime})} holds with $\kappa_n=[\log(1+\min\{x_n,\ell_n\})]^a(\log\ell_n)^{1-a}$.
\end{itemize}

Moreover, for Case (1), for Case (2) with $\min\{x_n,\ell_n\}=O(1)$ and for Case (3) with $\min\{x_n,\ell_n\}=O(1)$, when $a\ge 1$, and $\ell_n=O(1)$, when $0<a<1$, one has
\[
 T_{n,2}(\delta_{\mathbf{0}},\epsilon)\asymp p_{x_n}^{-1},\quad\forall \epsilon\in(0,B/\sqrt{2}),
\]
where $B=\min\{\inf_nA_n,\inf_nB_n\}/(A_1+B_1)$ and two sequences of positive reals, $a_n$ and $b_n$, satisfy $a_n\asymp b_n$ if $a_n=O(b_n)$ and $b_n=O(a_n)$.
\end{thm}

Now, let's consider the specific case of $p_i=i+1$, $x_n=\lfloor n^\alpha\rfloor$ with $\alpha\in[0,1)$ and $\ell_n=n-x_n+1$ and, for simplicity, assume that $A_1+B_1=1$ and $0<\inf_nA_n\le\sup_nA_n<1$. Clearly, this is the case of Theorem \ref{t-machinery}(2) with $a=b=1$. When $\alpha=0$, we are concerning the stability of components indexed from $1$ to $n$ and the result says that no $L^2$-cutoff exists and the $L^2$-mixing time is bounded above and below by universal positive constants. When $\alpha\in(0,1)$, we are concerning the stability of components indexed from $\lfloor n^{\alpha}\rfloor$ to $n$ (a large proportion of the case $\alpha=0$) and the result says that there is a $L^2$-cutoff with cutoff time $(\alpha\log n)/(2n^\alpha)$ that converges to $0$. It is interesting to see from the above discussion that the existence of $L^2$-cutoffs is sensitive at $\alpha=0$.

This paper is organized as follows. In Section \ref{s-ct}, we develop the framework of cutoffs for Laplace transforms in a different viewpoint from that in \cite{CSal10}. Compared with the heuristics introduced in \cite{CSal10}, the creation of Section \ref{s-ct} is more subtle and reveals more intrinsic profiles of cutoff phenomena. In Section \ref{s-l2}, the theoretical results in Section \ref{s-ct} are illustrated with reversible Markov chains and a comparison of cutoffs is made between the discrete time lazy versions and the continuous time chains. To see a practical application, we consider product chains in Section \ref{s-cexp} and derive a series of criteria on cutoffs and formulas on cutoff times, while some tricky techniques are addressed in the appendix.

{\bf Acknowledgement.} We thank Takashi Kumagai for his contribution in the development of the theoretical framework and the preparation of valuable examples. We also thank the referees for their careful reading and precious comments that enhance the readability of this article. The first author is partially supported by MOST grant MOST 104-2115-M-009-013-MY3 and by NCTS, Taiwan. The third author is supported by MOST grant MOST 104-2115-M-009-007 and NCTS, Taiwan.

\section{Cutoffs of Laplace transforms}\label{s-ct}

As the $L^2$-distances of reversible Markov chains can be expressed as generalized Laplace transforms in (\ref{eq-laplace}), we provide, in this section, a view point different from the framework in \cite{CSal10}, which leads to an improvement of the cutoff criterion in some aspect. For convenience, we limit the usage of notation $\mathcal{V}$ to the class of all non-decreasing and right-continuous functions $V$ on $(0,\infty)$ satisfying
\[
 \lim_{\lambda\ra 0^+}V(\lambda)=0,\quad \lim_{\lambda\ra\infty}V(\lambda)<\infty.
\]

Thereafter, for any two sequences of positive reals $a_n$ and $b_n$, we write $a_n=O(b_n)$ if $\sup_n\{a_n/b_n\}<\infty$ and write $a_n=o(b_n)$ if $a_n/b_n\ra 0$. In the case that $a_n=O(b_n)$ and $b_n=O(a_n)$, we simply say $a_n\asymp b_n$. When $a_n/b_n\ra 1$, we write $a_n\sim b_n$. Concerning the maximum and minimum of two reals $a$ and $b$, we write $a\vee b=\max\{a,b\}$ and $a\wedge b=\min\{a,b\}$.

\begin{defn}\label{d-laplace}
Let $V\in\mathcal{V}$.
\begin{itemize}
\item[(1)] The Laplace transform of $V$ is denoted by $\mathcal{L}_V$ and defined to be the following Lebesgue-Stieltjes integral
\[
 \mathcal{L}_V(t):=\int_{(0,\infty)}e^{-t\lambda}dV(\lambda),\quad\forall t\ge 0.
\]

\item[(2)] The mixing time of $\mathcal{L}_V$ is denoted and defined by
\[
 T_V(\epsilon):=\min\{t\ge 0|\mathcal{L}_V(t)\le\epsilon\},\quad\forall \epsilon>0.
\]
\end{itemize}
\end{defn}

Note that there is a one-to-one correspondence between $\mathcal{V}$ and the class of all finite Borel measures on $(0,\infty)$. For convenience, when $V\in\mathcal{V}$ and $E$ is a Borel set in $(0,\infty)$, we write $V(E)$ for the measurement of $E$ under the measure induced by $V$, which is the unique measure on $(0,\infty)$ satisfying $V((a,b])=V(b)-V(a)$ for all $0<a<b<\infty$. In particular, $V((0,b])=V(b)$ for $b>0$. Besides, it is easy to see from the definition of $\mathcal{L}_V$ that $\mathcal{L}_V(0)=V((0,\infty))$. As a result of the Lebesgue dominated convergence theorem, $\mathcal{L}_V$ is non-increasing and continuous on $[0,\infty)$ and vanishes at infinity.

\begin{lem}\label{l-ibp}
For $V\in\mathcal{V}$, $\mathcal{L}_V$ is strictly decreasing on $[0,\infty)$ and
\[
 \mathcal{L}_V(t)=t\int_{(0,\infty)}V(\lambda)e^{-t\lambda}d\lambda,\quad\forall t>0.
\]
\end{lem}
\begin{proof}
The first part is obvious from the definition of $\mathcal{L}_V$. For the second part, let $t>0$. Since $\lambda\mapsto e^{-t\lambda}$ is continuous, the integration by parts implies that, for $0<a<b<\infty$,
\[
 \int_{(a,b]}e^{-t\lambda}dV(\lambda)=e^{-bt}V(b)-e^{-at}V(a)
 +t\int_{(a,b]}V(\lambda)e^{-t\lambda}d\lambda.
\]
As $V$ is a bounded function vanishing at $0$, letting $a\ra 0$ and $b\ra\infty$ gives the desired identity.
\end{proof}

In the following, we introduce the concept of cutoffs for Laplace transforms, which should be regarded as a generalization of $L^2$-cutoffs for reversible Markov chains.

\begin{defn}\label{d-cutoff}
Let $(V_n)_{n=1}^\infty$ be a sequence in $\mathcal{V}$ and assume that
\[
 M:=\limsup_{n\ra\infty}\mathcal{L}_{V_n}(0)>0.
\]
The sequence $(\mathcal{L}_{V_n})_{n=1}^\infty$ is said to present
\begin{itemize}
\item[(1)] a pre-cutoff if there exist a sequence $t_n>0$ and positive constants $A<B$ such that
\[
 \lim_{n\ra\infty}\mathcal{L}_{V_n}(Bt_n)=0,\quad
 \liminf_{n\ra\infty}\mathcal{L}_{V_n}(At_n)>0.
\]

\item[(2)] a cutoff if there is a sequence $t_n>0$ such that
\[
 \lim_{n\ra\infty}\mathcal{L}_{V_n}(at_n)=\begin{cases}0&\forall a>1,\\M&\forall 0<a<1.\end{cases}
\]
\end{itemize}
In (2), $t_n$ is called a cutoff time.
\end{defn}
\begin{rem}
Note that a pre-cutoff is weaker than a cutoff but easy to be examined.
\end{rem}
\begin{rem}\label{r-cutoff}
One may check from the definition of cutoffs that, when $(\mathcal{L}_{V_n})_{n=1}^\infty$ has a cutoff, a sequence of positive reals $t_n$ is a cutoff time if and only if $t_n\sim T_{V_n}(\epsilon)$ for some $\epsilon>0$ and, further, either of them is equivalent to $t_n\sim T_{V_n}(\epsilon)$ for all $\epsilon>0$. Consequently, if $(\mathcal{L}_{V_n})_{n=1}^\infty$ has a cutoff, then $T_{V_n}(\epsilon)$ can be selected as a cutoff time for any $\epsilon>0$.
\end{rem}

The following theorem states the equivalence of pre-cutoffs and cutoffs, which is not correct in general.

\begin{thm}\label{t-precutoff}
Let $V_n\in\mathcal{V}$ and assume that $\limsup_n\mathcal{L}_{V_n}(0)>0$. Then, $(\mathcal{L}_{V_n})_{n=1}^\infty$ has a pre-cutoff if and only if $(\mathcal{L}_{V_n})_{n=1}^\infty$ has a cutoff.
\end{thm}

To prove the above theorem, the following lemma is required.

\begin{lem}\label{l-anal}\cite[Corollary 3.3]{CSal10}
Let $V_n\in\mathcal{V}$ and assume that $\sup_n\mathcal{L}_{V_n}(0)<\infty$. For any sequence $t_n>0$, the following functions
\[
 \overline{F}(a):=\limsup_{n\ra\infty}\mathcal{L}_{V_n}(at_n),\quad \underline{F}(a):=\liminf_{n\ra\infty}\mathcal{L}_{V_n}(at_n).
\]
are continuous on $(0,\infty)$. Further, if $\overline{F}(a)=0$ (resp. $\underline{F}(a)=0$) for some $a>0$, then $\overline{F}(a)=0$ (resp. $\underline{F}(a)=0$) for all $a>0$.
\end{lem}

\begin{rem}
It is worthwhile to remark from Lemma \ref{l-anal} that, in Definition \ref{d-cutoff}, $\mathcal{L}_{V_n}(0)\ra\infty$ is necessary for the existence of cutoffs.
\end{rem}

\begin{proof}[Proof of Theorem \ref{t-precutoff}]
The direction from cutoffs to pre-cutoffs is easy to see from the definition and we deal with the inverse direction in this proof. Let $M:=\limsup_n\mathcal{L}_{V_n}(0)$. Assume that $(\mathcal{L}_{V_n})_{n=1}^\infty$ has a pre-cutoff and let $t_n,A,B$ be as in Definition \ref{d-cutoff}(1). Set
$\alpha:=\min\{1,\liminf_n\mathcal{L}_{V_n}(At_n)\}$ and $s_n:=T_{V_n}(\alpha/2)$. In what follows, we show that $(\mathcal{L}_{V_n})_{n=1}^\infty$ has a cutoff with cutoff time $s_n$. From the definition of $s_n$ and the fact $\lim_n\mathcal{L}_{V_n}(Bt_n)=0$, one may choose $N>0$ such that $At_n\le s_n\le Bt_n$ for $n\ge N$. For $n\ge 1$, define
\[
 W_n(\lambda)=\int_{(0,\lambda]}e^{-s_n\eta}dV_n(\eta),\quad\forall \lambda\in(0,\infty).
\]
Clearly, $W_n\in\mathcal{V}$ and $dW_n(\lambda)=e^{-s_n\lambda}dV_n(\lambda)$, where the latter implies $\mathcal{L}_{W_n}(as_n)=\mathcal{L}_{V_n}((a+1)s_n)$ for $a\ge 0$ and then
\[
 \mathcal{L}_{W_n}(as_n)\le\mathcal{L}_{V_n}(Bt_n),\quad\forall a\ge B/A-1,\, n\ge N.
\]
As a result, the above observation yields that
\[
 \mathcal{L}_{W_n}(0)=\mathcal{L}_{V_n}(s_n)=\alpha/2,\quad\forall n\ge N,\quad \limsup_{n\ra\infty}\mathcal{L}_{W_n}((B/A-1)s_n)=0.
 \]
By Lemma \ref{l-anal}, we achieve the result of $\lim_n\mathcal{L}_{V_n}(bs_n)=0$ for all $b>1$.

To prove the desired cutoff, it remains to show that $\lim_n\mathcal{L}_{V_n}(bs_n)=\infty$ for $b\in(0,1)$. Assume the inverse that there is $b_0\in(0,1)$ and an increasing sequence $k_n$ in $\mathbb{N}$ such that $\sup_n\mathcal{L}_{V_{k_n}}(b_0s_{k_n})<\infty$. As before, we define
\[
 U_n(\lambda)=\int_{(0,\lambda]}e^{-b_0s_n\eta}dV_n(\eta),\quad\forall \lambda\in(0,\infty).
\]
Observe that $dU_n(\lambda)=e^{-b_0s_n\lambda}dV_n(\lambda)$. This implies $\mathcal{L}_{U_n}(as_n)=\mathcal{L}_{V_n}((a+b_0)s_n)$ and, thus,
\[
 \sup_{n\ge 1}\mathcal{L}_{U_{k_n}}(0)<\infty,\quad \limsup_{n\ra\infty}\mathcal{L}_{U_{k_n}}((B/A-b_0)s_{k_n})
 \le\limsup_{n\ra\infty}\mathcal{L}_{V_{k_n}}(Bt_{k_n})=0.
\]
By Lemma \ref{l-anal}, $\mathcal{L}_{U_{k_n}}(as_{k_n})\ra 0$ for all $a>0$, which contradicts the fact $\mathcal{L}_{U_n}((1-b_0)s_n)=\alpha/2>0$ for $n\ge N$. This proves that $(\mathcal{L}_{V_n})_{n=1}^\infty$ has a cutoff.
\end{proof}

Next, we provide criteria to judge the existence of cutoffs and formulas to characterize cutoff times. First of all, we need the following notations to state it. For $V\in\mathcal{V}$ and $c\in(0,\mathcal{L}_V(0))$, set
\[
   \lambda_V(c):=\inf\{\lambda|V(\lambda)>c\},\quad \tau_V(c):=\sup_{\lambda\ge\lambda_V(c)}
   \left\{\frac{\log(1+V(\lambda))}{\lambda}\right\}.
\]
The next theorem contains the key technique in this article that supports Theorems \ref{t-main}, \ref{t-l2cutcts} and \ref{t-l2cutdis}.

\begin{thm}\label{t-ltcutoff}
Consider a sequence $(V_n)_{n=1}^\infty$ in $\mathcal{V}$ and
assume that $\mathcal{L}_{V_n}(0)\ra\infty$. The following statements are equivalent.
\begin{itemize}
\item[(1)] $(\mathcal{L}_{V_n})_{n=1}^\infty$ has a cutoff.

\item[(2)] For all $\epsilon>0$ and $c>0$, $T_{V_n}(\epsilon)\lambda_{V_n}(c)\ra\infty$.

\item[(3)] There exists $\epsilon>0$ such that $T_{V_n}(\epsilon)\lambda_{V_n}(c)\ra\infty$ for all $c>0$.

\item[(4)] For all $c>0$, $\tau_{V_n}(c)\lambda_{V_n}(c)\ra\infty$.

\item[(5)] For all $\widetilde{c}>0$ and $c>0$, $\tau_{V_n}(\widetilde{c})\lambda_{V_n}(c)\ra\infty$.

\item[(6)] There is $\widetilde{c}>0$ such that $\tau_{V_n}(\widetilde{c})\lambda_{V_n}(c)\ra\infty$ for all $c>0$.
\end{itemize}
In particular, if $(\mathcal{L}_{V_n})_{n=1}^\infty$ has a cutoff, then $\tau_{V_n}(c)$ is a cutoff time for any $c>0$. Furthermore, one has
\begin{equation}\label{eq-mixbound1}
 |T_{V_n}(\epsilon)-T_{V_n}(\delta)|=O(1/\lambda_{V_n}(c)),\quad\forall \epsilon,\delta,c\in(0,\infty),
\end{equation}
and
\begin{equation}\label{eq-mixbound2}
 |T_{V_n}(\epsilon)-\tau_{V_n}(c)|
 =O\left(\sqrt{\tau_{V_n}(c)/\lambda_{V_n}(c)}\right),
 \quad\forall \epsilon,c\in(0,\infty).
\end{equation}
\end{thm}

\begin{rem}
Based on the assumption of $\mathcal{L}_{V_n}(0)\ra\infty$, there exists, for any $c>0$, a constant $N$ such that $\lambda_{V_n}(c)$ is defined for $n\ge N$.
\end{rem}
\begin{rem}
We would like to emphasize that, in Theorem \ref{t-ltcutoff}, conditions (3), (4) and (6) are useful in proving the existence of cutoffs, while conditions (2) and (5) make the disproof of cutoffs easier.
\end{rem}

Before proving Theorem \ref{t-ltcutoff}, we would like to highlight the fact that, when proving or disproving cutoffs with conditions (3) and (4), one should pay attention to the corresponding limits with small $c$. This is given by the following lemma.

\begin{lem}\label{l-cond34}
Let $(V_n)_{n=1}^\infty$ be a sequence in $\mathcal{V}$ satisfying $\mathcal{L}_{V_n}(0)\ra\infty$. For $c'>c$, one has
\[
 T_{V_n}(\epsilon)\lambda_{V_n}(c)\ra\infty\quad\Ra\quad T_{V_n}(\epsilon)\lambda_{V_n}(c')\ra\infty.
\]
and
\[
 \tau_{V_n}(c)\lambda_{V_n}(c)\ra\infty\quad\Ra\quad \tau_{V_n}(c')\lambda_{V_n}(c')\ra\infty.
\]
\end{lem}
\begin{proof}
The first part is a corollary of the observation that $\lambda_{V_n}(c_1)\le\lambda_{V_n}(c_2)$ for $0<c_1<c_2<\mathcal{L}_{V_n}(0)$. To see the second part, suppose $\tau_{V_n}(c)\lambda_{V_n}(c)\ra\infty$. By Lemma \ref{l-tvc} (See the following), there is $\gamma_n\ge\lambda_{V_n}(c)$ such that
\[
 \tau_{V_n}(c)=\sup_{\lambda\ge\lambda_{V_n}(c)}\left\{\frac{\log(1+V_n(\lambda))}
 {\lambda}\right\}=\frac{\log(1+V_n(\gamma_n))}{\gamma_n}.
\]
This implies
\[
 V_n(\gamma_n)\ge \log(1+V_n(\gamma_n))=\tau_{V_n}(c)\gamma_n\ge\tau_{V_n}(c)\lambda_{V_n}(c)\ra\infty.
\]
Consequently, for any $c'>c$, there is $N=N(c')$ such that $\tau_{V_n}(c')=\tau_{V_n}(c)$ for $n\ge N$ and this leads to $\tau_{V_n}(c')\lambda_{V_n}(c')\ge\tau_{V_n}(c)\lambda_{V_n}(c)\ra\infty$.
\end{proof}

In the remaining of this section, we focus on proving Theorem \ref{t-ltcutoff} and, first, create two lemmas and one proposition.

\begin{lem}\label{l-tvc}
Fix $V\in\mathcal{V}$ and let $F(\lambda)=\lambda^{-1}\log(1+V(\lambda))$ for $\lambda\in(0,\infty)$. Then, $F$ is right continuous with left limit and satisfying
\[
 \lim_{\lambda<c,\lambda\ra c}F(\lambda)\le \lim_{\lambda>c,\lambda\ra c}F(\lambda).
\]
In particular, for $c\in(0,\mathcal{L}_V(0))$, there is $\gamma\ge \lambda_V(c)$ such that $\tau_V(c)=\gamma^{-1}\log(1+V(\gamma))$.
\end{lem}
\begin{proof}
The right-continuity and limiting behavior of $F$ is obvious from its definition. Next, we deal with the second part. Let $c\in(0,\mathcal{L}_V(0))$. Clearly, $\lambda_V(c)\in(0,\infty)$. By restricting the domain of $F$ to $[\lambda_V(c),\infty)$, the function $F$ is bounded and vanishes at infinity. This implies that there is a bounded monotone sequence $u_n\in[\lambda_V(c),\infty)$ such that $F(u_n)\ra\tau_V(c)$. If $\gamma$ is the limit of $u_n$, then the first part of this lemma implies $\tau_V(c)=\lim_nF(u_n)\le F(\gamma)\le \tau_V(c)$ as desired.
\end{proof}

\begin{lem}\label{l-ltcomp}
Let $V\in\mathcal{V}$ and $\epsilon,c,c_1,c_2$ be constants in $(0,\mathcal{L}_V(0))$.
\begin{itemize}
\item[(1)] $\mathcal{L}_V(\tau_V(c))\ge c/(1+c)$ and, for $s>0$,
\[
 \mathcal{L}_V(\tau_V(c)+s)\le c+\frac{\tau_V(c)+s}{s e^{s\lambda_V(c)}}.
\]

\item[(2)] $\mathcal{L}_V(T_V(\epsilon))=\epsilon$ and, for $r\ge 0$, $s>0$ and $c_1<c_2$,
\[
 \mathcal{L}_V(T_V(\epsilon)+r+s)\le c_1+c_2e^{-(T_V(\epsilon)+r+s)\lambda_V(c_1)}
 +\frac{\epsilon(T_V(\epsilon)+r+s)}{(T_V(\epsilon)+s)e^{r\lambda_V(c_2)}}.
\]
\end{itemize}
\end{lem}
\begin{proof}
The proof is a little lengthy and delegated to the appendix.
\end{proof}

\begin{prop}\label{p-ltcomp}
Let $V\in\mathcal{V}$,  $\epsilon,c,c_1,c_2$ be constants in $(0,\mathcal{L}_V(0))$ and $\alpha=\sqrt{\tau_V(c)\lambda_V(c)}$. Then,

\begin{equation}\label{eq-mixing1}
 \left(\frac{\alpha}{\alpha+A}\right)T_V\left(c+\frac{A+\alpha}{Ae^{A\alpha}}\right)
 \le\tau_V(c)\le T_V\left(\frac{c}{1+c}\right),\quad\forall A>0,
\end{equation}
and
\begin{equation}\label{eq-mixing2}
 T_V\left(c_1+c_2e^{-T_V(\epsilon)\lambda_V(c_1)}+2\epsilon e^{-B}\right)\le T_V(\epsilon)+\frac{2B}{\lambda_V(c_2)},\quad\forall B>0.
\end{equation}
In particular, one has
\begin{equation}\label{eq-bm0}
 \tau_V(2\delta)\le T_V(\delta)\le\frac{6}{\delta^2}\tau_V\left(\frac{\delta}{2}\right),\quad \forall 0<\delta<\frac{\mathcal{L}_V(0)\wedge 1}{2}.
\end{equation}
\end{prop}
\begin{proof}
(\ref{eq-mixing1}) follows immediately from Lemma \ref{l-ltcomp}(1) with $s=A\tau_V(c)/\alpha$. For (\ref{eq-mixing2}), the replacement of $s=r=B/\lambda_V(c_2)$ in Lemma \ref{l-ltcomp}(2) yields
\[
 \mathcal{L}_V(T_V(\epsilon)+2B/\lambda_V(c_2))\le c_1+c_2e^{-T_V(\epsilon)\lambda_V(c_1)}+2\epsilon e^{-B},
\]
which leads to the desired inequality.

Next, we prove (\ref{eq-bm0}). From the definitions of $\lambda_V(c)$ and $\tau_V(c)$, it is easy to see that $\alpha\ge\sqrt{\log(1+c)}$. As a result, when $A=1/[c\sqrt{\log(1+c)}]$, one has
\[
 \frac{1+\alpha/A}{e^{\alpha A}}\le\frac{1+\alpha/A}{1+\alpha A}
 \le\frac{c[1+c\log(1+c)]}{1+c}\le c,\quad\forall 0<c<\mathcal{L}_V(0)\wedge 1.
\]
By (\ref{eq-mixing1}), this implies
\[
 \frac{c\log(1+c)}{c\log(1+c)+1}T_V(2c)\le\tau_V(c)\le T_V\left(\frac{c}{1+c}\right)\le T_V(c/2).
\]
Replacing $c$ with $\delta/2$ and $2\delta$ in the first and second inequalities, we obtain
\[
 \tau_V(2\delta)\le T_V(\delta)\le\frac{\delta\log(1+\delta/2)+2}
 {\delta\log(1+\delta/2)}\tau_V(\delta/2)\le\frac{6}{\delta^2}
 \tau_V(\delta/2),
\]
for $0<\delta<(\mathcal{L}_V(0)\wedge 1)/2$, where the last inequality uses the fact of $\log(1+u)\ge u/(1+u)$ for $u>-1$.
\end{proof}

\begin{proof}[Proof of Theorem \ref{t-ltcutoff}]
We first show the equivalence of (1), (2) and (3). Assume that $(\mathcal{L}_{V_n})_{n=1}^\infty$ has a cutoff and let $\epsilon>0$  and $c>0$. By Remark \ref{r-cutoff}, $T_{V_n}(\epsilon)$ can be a cutoff time and this implies
\[
 \mathcal{L}_{V_n}(2T_{V_n}(\epsilon))\ge\int_{(0,\lambda_{V_n}(c)]}
 e^{-2T_{V_n}(\epsilon)\lambda}dV_n(\lambda)\ge ce^{-2T_{V_n}(\epsilon)\lambda_{V_n}(c)}.
\]
Letting $n\ra\infty$ yields $T_{V_n}(\epsilon)\lambda_{V_n}(c)\ra\infty$. This proves (1)$\Ra$(2), while (2)$\Ra$(3) is obvious.

Next, we assume (3) and let $\epsilon>0$ be a constant such that $T_{V_n}(\epsilon)\lambda_{V_n}(c)\ra\infty$ for all $c>0$. By Lemma \ref{l-ibp}, one has
\[
  \mathcal{L}_{V_n}(aT_{V_n}(\epsilon))=aT_{V_n}(\epsilon)
  \int_{(0,\infty)}V_n(\lambda)e^{-aT_{V_n}(\epsilon)\lambda}d\lambda,\quad\forall a>0.
\]
This implies that, for $a>1$,
\begin{align}
 \mathcal{L}_{V_n}(aT_{V_n}(\epsilon))&\le c+aT_{V_n}(\epsilon)
 \int_{[\lambda_{V_n}(c),\infty)}V_n(\lambda)e^{-aT_{V_n}(\epsilon)\lambda}
 d\lambda\notag\\
 &\le c+ae^{(1-a)T_{V_n}(\epsilon)\lambda_{V_n}(c)}T_{V_n}(\epsilon)
 \int_{[\lambda_{V_n}(c),\infty)}V_n(\lambda)
 e^{-T_{V_n}(\epsilon)\lambda}d\lambda\notag\\
 &\le c+a\epsilon e^{(1-a)T_{V_n}(\epsilon)\lambda_{V_n}(c)},\notag
\end{align}
and, similarly, for $a\in(0,1)$,
\[
 \mathcal{L}_{V_n}(T_{V_n}(\epsilon))\le c+a^{-1}e^{(a-1)T_{V_n}(\epsilon)\lambda_{V_n}(c)}\mathcal{L}_{V_n}(aT_{V_n}(\epsilon)).
\]
Since $\mathcal{L}_{V_n}(0)\ra \infty$, there is $N>0$ such tat $T_{V_n}(\epsilon)>0$ for $n\ge N$. This implies $\mathcal{L}_{V_n}(T_{V_n}(\epsilon))=\epsilon$ for $n\ge N$ and
\[
 \mathcal{L}_{V_n}(aT_{V_n}(\epsilon))\ge(\epsilon-c)
 ae^{(1-a)T_{V_n}(\epsilon)\lambda_{V_n}(c)},\quad\forall a\in(0,1),\,n\ge N.
\]
As a consequence, we obtain that, for $a>1$,
\[
 \limsup_{n\ra\infty}\mathcal{L}_{V_n}(aT_{V_n}(\epsilon))\le\limsup_{c\ra 0}\limsup_{n\ra\infty}\left(c+a\epsilon e^{(1-a)T_{V_n}(\epsilon)\lambda_{V_n}(c)}\right)=0,
\]
and, for $a\in(0,1)$ and $c\in(0,\epsilon)$,
\[
 \liminf_{n\ra\infty}\mathcal{L}_{V_n}(aT_{V_n}(\epsilon))
 \ge\liminf_{n\ra\infty}(\epsilon-c)
 ae^{(1-a)T_{V_n}(\epsilon)\lambda_{V_n}(c)}=\infty.
\]
This proves that $(\mathcal{L}_{V_n})_{n=1}^\infty$ has a cutoff.

Now, we prove the equivalence of (1)-(6). First, consider (2)$\Ra$(4) and set
\begin{equation}\label{eq-alpha}
 \alpha_n(c)=\sqrt{\tau_{V_n}(c)\lambda_{V_n}(c)}, \quad\forall c>0.
\end{equation}
By applying the first inequality of (\ref{eq-mixing1}) to $V_n$ with $A=\alpha_n(c)$, we obtain $\tau_{V_n}(c)\ge T_{V_n}(c+2)/2$. Based on the assumption of (2), this implies $\tau_{V_n}(c)\lambda_{V_n}(c)\ge T_{V_n}(c+2)\lambda_{V_n}(c)/2\ra\infty$, which proves (4). (5)$\Ra$(6) is obvious, while (6)$\Ra$(3) is given by the second inequality of (\ref{eq-mixing1}). To finish the proof of equivalence, it remains to show that (4)$\Ra$(5). Suppose that (4) holds and let $c_1,c_2$ be positive constants. For convenience, we set $F_n(\lambda)=\lambda^{-1}\log(1+V_n(\lambda))$. Since $\mathcal{L}_{V_n}(0)\ra\infty$, one may select $N>0$ such that $c_1\vee c_2<\mathcal{L}_{V_n}(0)$ for $n\ge N$. By Lemma \ref{l-tvc}, there are $\gamma_{i,n}\ge\lambda_{V_n}(c_i)$ with $i\in\{1,2\}$ such that
\begin{equation}\label{eq-gamma}
 \tau_{V_n}(c_i)=F_n(\gamma_{i,n})
 =\sup_{\lambda\ge\lambda_{V_n}(c_i)}F_n(\lambda),\quad\forall i=1,2.
\end{equation}
The first identity in (\ref{eq-gamma}) implies
\[
 \tau_{V_n}(c_i)\lambda_{V_n}(c_i)\le\log(1+V_n(\gamma_{i,n})),\quad\forall n\ge N,
\]
and, by the assumption of (4), $V_n(\gamma_{i,n})\ra\infty$ for $i=1,2$. As a result, we may refine $N$ such that $V_n(\gamma_{1,n})\wedge V_n(\gamma_{2,n})\ge c_1\vee c_2$ for $n\ge N$. By (\ref{eq-gamma}), this implies
\[
 \tau_{V_n}(c_i)=\sup\{F_n(\lambda)|\lambda\ge\lambda_{V_n}(c_1\vee c_2)\},\quad
 \forall i\in\{1,2\},\,n\ge N,
\]
and, hence, $\tau_{V_n}(c_1)=\tau_{V_n}(c_2)$ for $n\ge N$. Consequently, we obtain that both $\tau_{V_n}(c_1)\lambda_{V_n}(c_2)$ and $\tau_{V_n}(c_2)\lambda_{V_n}(c_1)$ tend to infinity, as desired in (5).

In the end, we derive a cutoff time and the bounds in (\ref{eq-mixbound1})-(\ref{eq-mixbound2}). Suppose that $(\mathcal{L}_{V_n})_{n=1}^\infty$ has a cutoff. By Remark \ref{r-cutoff}, one has $T_{V_n}(\epsilon)\sim T_{V_n}(\delta)$ for all $\epsilon,\delta\in(0,\infty)$ and, referring to the setting in (\ref{eq-alpha}), (4) implies $\alpha_n(c)\ra\infty$ for all $c>0$. Applying (\ref{eq-mixing1}) with $A=1$ and the fact of $e^x\ge 1+x$, we obtain
\begin{equation}\label{eq-alphan}
 \frac{\alpha_n(c)}{\alpha_n(c)+1}T_{V_n}(c+1)\le\tau_{V_n}(c)\le T_{V_n}(c/(1+c)),
\end{equation}
for all $n,c$ satisfying $\mathcal{L}_{V_n}(0)>c$. As $\mathcal{L}_{V_n}(0)\ra\infty$, letting $n\ra\infty$ yields $\tau_{V_n}(c)\sim T_{V_n}(c/(1+c))$ and, by Remark \ref{r-cutoff}, $\tau_{V_n}(c)$ is a cutoff time for all $c>0$.

For (\ref{eq-mixbound1}), let $\epsilon>\delta>0$ and $c>0$. Since $(\mathcal{L}_{V_n})_{n=1}^\infty$ has a cutoff, Theorem \ref{t-ltcutoff}(2) implies $T_{V_n}(\epsilon) \lambda_{V_n}(\delta/2)\ra\infty$. By applying (\ref{eq-mixing2}) with $V=V_n$, $c_1=\delta/2$ and $c_2=c$ and using the fact of $\mathcal{L}_{V_n}(0)\ra\infty$, one may select $B>0$ and $N>0$ such that, for $n\ge N$,
\[
 T_{V_n}(\delta)\le T_{V_n}\left(\frac{\delta}{2}+ce^{-T_{V_n}(\epsilon)\lambda_{V_n}(\delta/2)}+2\epsilon e^{-B}\right)\le T_{V_n}(\epsilon)+2B/\lambda_{V_n}(c).
\]
As it is clear from the definition of $T_{V_n}$ that $T_{V_n}(\delta)\ge T_{V_n}(\epsilon)$, the above inequalities lead to (\ref{eq-mixbound1}).

To see (\ref{eq-mixbound2}), let $\epsilon,c\in(0,\infty)$ and write
\[
 |\tau_{V_n}(c)-T_{V_n}(\epsilon)|\le|\tau_{V_n}(c)-T_{V_n}(c/(c+1))|
 +|T_{V_n}(c/(c+1))-T_{V_n}(\epsilon)|.
\]
Note that, by (\ref{eq-alphan}), if $\mathcal{L}_{V_n}(0)>c$, then
\begin{align}
 |\tau_{V_n}(c)-T_{V_n}(c/(c+1))|\le&|T_{V_n}(c/(1+c))-T_{V_n}(c+1)|\notag\\
 &\qquad+\frac{T_{V_n}(c+1)/\alpha_n(c)}{1+1/\alpha_n(c)}.\notag
\end{align}
Assuming that $(\mathcal{L}_{V_n})_{n=1}^\infty$ has a cutoff, (\ref{eq-mixbound1}) gives
\[
 |T_{V_n}(\widetilde{c})-T_{V_n}(\epsilon)|=O(1/\lambda_{V_n}(c)),\quad\forall \widetilde{c}>0,
\]
and, by the triangle inequality, this implies
\[
 |T_{V_n}(c/(1+c))-T_{V_n}(c+1)|=O(1/\lambda_{V_n}(c)).
\]
As a result of Theorem \ref{t-ltcutoff}(4), $\alpha_n(c)\ra\infty$ and this is equivalent to $1/\lambda_{V_n}(c)=o(\sqrt{\tau_{V_n}(c)/\lambda_{V_n}(c)})$. Since $\tau_{V_n}(c)$ is a cutoff time, Remark \ref{r-cutoff} implies $T_{V_n}(c+1)\sim\tau_{V_n}(c)$ and this leads to
\[
 \frac{T_{V_n}(c+1)/\alpha_n(c)}{1+1/\alpha_n(c)}\sim \frac{\tau_{V_n}(c)}{\alpha_n(c)}
 =\sqrt{\frac{\tau_{V_n}(c)}{\lambda_{V_n}(c)}},
\]
as desired.
\end{proof}

\section{Cutoff of reversible Markov chains}\label{s-l2}

The goal of this section is two-fold. In the first subsection, we derive criteria for $L^2$-cutoffs and formulas for $L^2$-cutoff times using the results in Section \ref{s-ct}. In the second subsection, we provide a comparison of $L^2$-cutoffs between continuous time chains and lazy discrete time chains. Note that the theory developed in Section \ref{s-ct} is immediately applicable for the continuous time case. In the discrete time case, one should be aware that the time sequence is integer-valued but there is no big difference in concluding similar results due to the assumption that the $L^2$-mixing time tends to infinity.

As in the introduction, we write $\mathcal{F}$ for a family of irreducible and reversible finite Markov chains. In the discrete time case, it means $\mathcal{F}=(\mu_n,\mathcal{S}_n,K_n,\pi_n)_{n=1}^\infty$ and, in the continuous time case, one has $\mathcal{F}=(\mu_n,\mathcal{S}_n,L_n,\pi_n)_{n=1}^\infty$. In either case, we use $d_{n,2}(\mu_n,\cdot)$ and $T_{n,2}(\mu_n,\cdot)$ to denote the $L^2$-distance and the $L^2$-mixing time of the $n$th chain in $\mathcal{F}$.

\subsection{$L^2$-cutoffs for reversible Markov chains}
One can see from (\ref{eq-l2cutoffdef}) that, to identify a $L^2$-cutoff, either a precise estimation of the $L^2$-cutoff time is made or a sophisticated computation of the $L^2$-mixing time is required. Instead of dealing with the existence of a cutoff directly, it could be more efficient to explore the existence of a pre-cutoff, which is a necessary condition for a cutoff, in advance. In the discrete time case, we say that $\mathcal{F}$ has a $L^2$-pre-cutoff if there are positive constants $A<B$ and a sequence of positive reals $(t_n)_{n=1}^\infty$ such that
\[
 \limsup_{n\ra\infty}d_{n,2}(\mu_n,\lceil Bt_n\rceil)=0,\quad \liminf_{n\ra\infty}d_{n,2}(\mu_n,\lfloor At_n\rfloor)>0.
\]
In the continuous time case, the $L^2$-pre-cutoff is similarly defined by removing $\lceil\cdot\rceil$ and $\lfloor\cdot\rfloor$.

It can be seen from the above definition and (\ref{eq-l2cutoffdef}) that, for families of continuous time chains, $\liminf_n\pi_n(|\mu_n/\pi_n|^2)>1$ is necessary for the existence of a $L^2$-pre-cutoff and $\lim_n\pi_n(|\mu_n/\pi_n|^2)=\infty$ is necessary for the presence of a $L^2$-cutoff. For families of discrete time chains, we consider the specific case that $T_{n,2}(\mu_n,\epsilon_0)\ra\infty$ for some $\epsilon_0\in(0,\infty)$. By following the definition, if $\mathcal{F}$ has a $L^2$-pre-cutoff, then $\liminf_n\pi_n(|\mu_nK_n/\pi_n|^2)>1$; if $\mathcal{F}$ presents a $L^2$-cutoff, then $\lim_n\pi_n(|\mu_nK_n/\pi_n|^2)=\infty$. It is clear that both conclusions are more rigid than those necessary conditions in the continuous time case. A reason why we consider $\pi_n(|\mu_nK_n/\pi_n|^2)$ instead of $\pi_n(|\mu_n/\pi_n|^2)$ is that, by the first identity in (\ref{eq-l2dis}), when a chain starts evolving, those eigenvectors corresponding to eigenvalue $0$ play no roles in the $L^2$-distance and thus should be discarded. In other words, when concerning a discrete time chain, say $(\mu,\mathcal{S},K,\pi)$, it is more meaningful to consider the time-shifted chain $(\mu K,\mathcal{S},K,\pi)$ instead.

By (\ref{eq-laplace}) and (\ref{eq-V}), the following three theorems are immediate applications of Theorems \ref{t-precutoff}-\ref{t-ltcutoff} to finite Markov chains. The first theorem establishes the equivalence of $L^2$-cutoffs and $L^2$-pre-cutoff, which can fail in general, say in the total variation and in separation.

\begin{thm}\label{t-l2precutoff}
Let $\mathcal{F}$ be a family of irreducible and reversible finite Markov chains.
\begin{itemize}
\item[(1)] For the continuous time case, assume that $\liminf_n\pi_n(|\mu_n/\pi_n|^2)>1$. Then, $\mathcal{F}$ has a $L^2$-cutoff if and only if $\mathcal{F}$ has a $L^2$-pre-cutoff.

\item[(2)] For the discrete time case, assume that $\liminf_n\pi_n(|\mu_nK_n/\pi_n|^2)>1$ and $T_{n,2}(\mu_n,\epsilon_0)\ra\infty$ for some $\epsilon_0\in(0,\infty)$. Then, $\mathcal{F}$ has a $L^2$-cutoff if and only if $\mathcal{F}$ has a $L^2$-pre-cutoff.
\end{itemize}
\end{thm}

To state the other two theorems, we need the following notations. Let $(\mu,\mathcal{S},L,\pi)$ be an irreducible and reversible continuous time finite Markov chain and $\lambda_0=0<\lambda_1\le\cdots\le\lambda_{|\mathcal{S}|-1}$ be eigenvalues of $-L$ with $L^2(\pi)$-orthonormal right eigenvectors $\phi_0=\mathbf{1}$, $\phi_1$,..., $\phi_{|\mathcal{S}|-1}$. For $c>0$, define
\begin{equation}\label{eq-jc}
 j(c):=\min\left\{j\ge 1\bigg|\sum_{i=1}^j|\mu(\phi_i)|^2>c\right\},
\end{equation}
and
\begin{equation}\label{eq-tauc}
 \tau(c):=\max_{j\ge j(c)}\left\{\frac{\log\left(1+\sum_{i=1}^j|\mu(\phi_i)|^2\right)}{2\lambda_j}\right\}.
\end{equation}
For the discrete time chain $(\mu,\mathcal{S},K,\pi)$, we define $j(c),\tau(c)$ by following (\ref{eq-jc})-(\ref{eq-tauc}) under the replacement of $\lambda_i$ with $-\log|\beta_i|$, where $\beta_i$'s and $\phi_i$'s are eigenvalues and $L^2(\pi)$-orthonormal right eigenvectors of $K$ satisfying $\beta_0=1>|\beta_1|\ge\cdots\ge|\beta_{|\mathcal{S}|-1}|$ and $1/\infty:=0$.

\begin{thm}\label{t-l2cutcts}
Consider a family of irreducible and reversible continuous time finite Markov chains $\mathcal{F}=(\mu_n,\mathcal{S}_n,L_n,\pi_n)_{n=1}^\infty$. Let $0<\lambda_{n,1}<\cdots<\lambda_{n,|\mathcal{S}_n|-1}$ be the eigenvalues of $-L_n$ and $j_n(c),\tau_n(c)$ be the constants in \textnormal{(\ref{eq-jc})-(\ref{eq-tauc})}. Assume that $\pi_n(|\mu_n/\pi_n|^2)\ra\infty$. Then, the following are equivalent.
\begin{itemize}
\item[(1)] $\mathcal{F}$ has a $L^2$-cutoff.

\item[(2)] For all $\epsilon>0$ and $c>0$, $T_{n,2}(\mu_n,\epsilon)\lambda_{n,j_n(c)}\ra\infty$.

\item[(3)] There is $\epsilon>0$ such that $T_{n,2}(\mu_n,\epsilon)\lambda_{n,j_n(c)}\ra\infty$ for all $c>0$.

\item[(4)] For all $c>0$, $\tau_n(c)\lambda_{n,j_n(c)}\ra\infty$.

\item[(5)] For all $\widetilde{c}>0$ and $c>0$, $\tau_n(\widetilde{c})\lambda_{n,j_n(c)}\ra\infty$.

\item[(6)] There is $\widetilde{c}>0$ such that $\tau_n(\widetilde{c})\lambda_{n,j_n(c)}\ra\infty$ for all $c>0$.
\end{itemize}
Further, if $\mathcal{F}$ has a $L^2$-cutoff, then $\tau_n(c)$ is a cutoff time for any $c>0$ and
\[
 |T_{n,2}(\mu_n,\epsilon)-T_{n,2}(\mu_n,\delta)|=O\left(1/\lambda_{n,j_n(c)})\right),
 \quad\forall \epsilon,\delta,c\in(0,\infty),
\]
and
\begin{equation}\label{eq-l2mixingtime}
 |T_{n,2}(\mu_n,\epsilon)-\tau_n(c)|=O\left(\sqrt{\tau_n(c)/\lambda_{n,j_n(c)}}\right),\quad \forall \epsilon,c\in(0,\infty).
\end{equation}
\end{thm}

\begin{thm}\label{t-l2cutdis}
Consider a family of irreducible and reversible discrete time finite Markov chains $\mathcal{F}=(\mu_n,\mathcal{S}_n,K_n,\pi_n)_{n=1}^\infty$. Let $\{1\}\cup\{\beta_{n,i}:i\ge 1\}$ be the eigenvalues of $K_n$ satisfying $|\beta_{n,1}|\ge\cdots\ge|\beta_{n,|\mathcal{S}_n|-1}|$ and $j_n(c),\tau_n(c)$ be the constants in \textnormal{(\ref{eq-jc})-(\ref{eq-tauc})} with $\lambda_{n,i}=-\log|\beta_{n,i}|$. Assume that $T_{n,2}(\mu_n,\epsilon_0)\ra\infty$ for some $\epsilon_0>0$ or $\tau_n(c)\ra\infty$ for some $c>0$. Assume further that $\pi_n(|\mu_nK_n/\pi_n|^2)\ra\infty$. Then, the equivalences in Theorem \ref{t-l2cutcts} also hold in this case. Further, if $\mathcal{F}$ has a $L^2$-cutoff, then $\tau_n(c)$ is a cutoff time for any $c>0$ and
\begin{equation}\label{eq-max1}
 |T_{n,2}(\mu_n,\epsilon)-T_{n,2}(\mu_n,\delta)|=O\left(\max\{1,1/\lambda_{n,j_n(c)}\}\right),
 \quad\forall \epsilon,\delta,c\in(0,\infty),
\end{equation}
and
\begin{equation}\label{eq-max2}
 |T_{n,2}(\mu_n,\epsilon)-\tau_n(c)|
 =O\left(\max\left\{1,\sqrt{\tau_n(c)/\lambda_{n,j_n(c)}}\right\}\right),
 \quad \forall \epsilon,c\in(0,\infty).
\end{equation}
\end{thm}

\begin{rem}
Note that the mixing time of a discrete time chain is integer-valued and this results in the difference of (\ref{eq-max1})-(\ref{eq-max2}) from those corresponding identities in Theorem \ref{t-l2cutcts}.
\end{rem}

\begin{rem}
In Theorem \ref{t-l2cutcts}, the bound on the difference of $L^2$-mixing times say that, in the continuous time case, if the $L^2$-mixing time is selected as a $L^2$-cutoff time, then the cutoff window is at most $1/\lambda_{n,j_n(c)}$; if $\tau_n(c)$ is chosen as a $L^2$-cutoff time, then the cutoff window should be less than $\sqrt{\tau_n(c)/\lambda_{n,j_n(c)}}$, which is of order bigger than $1/\lambda_{n,j_n(c)}$. For the discrete time case, Theorem \ref{t-l2cutdis} provides a somewhat difference conclusion in (\ref{eq-max1})-(\ref{eq-max2}) due to the restriction of integer-valued times. The readers are referred to \cite{CSal08,CSal10} for a definition and more information of cutoff windows.
\end{rem}

As Theorems \ref{t-l2cutcts} and \ref{t-l2cutdis} provide criteria to inspect cutoffs and compute cutoff times, the following proposition supplies definite bounds on mixing times using (\ref{eq-tauc}), which is crucial to a family without cutoff.

\begin{prop}\label{p-bm}
Let $(\mu,\mathcal{S},L,\pi)$ and $(\mu,\mathcal{S},K,\pi)$ be irreducible and reversible finite Markov chains and $j(c),\tau(c)$ be the constants in \textnormal{(\ref{eq-jc})-(\ref{eq-tauc})}. Let $T_2(\mu,\cdot)$ be the $L^2$-mixing time and set $\alpha(c)=\sqrt{\tau(c)\lambda_{j(c)}}$.
\begin{itemize}
\item[(1)] For the continuous time case, one has, for $0<c<\pi(|\mu/\pi|^2)-1$ and $A>0$,
\begin{equation}\label{eq-bmcts1}
 \frac{\alpha(c)}{\alpha(c)+A}T_2\left(\mu,\sqrt{c+\frac{A+\alpha(c)}{Ae^{\alpha(c) A}}}\right)\le\tau(c)\le T_2\left(\mu,\sqrt{\frac{c}{1+c}}\right).
\end{equation}
In particular, for $0<\epsilon<\sqrt{[\pi(|\mu/\pi-1|^2)\wedge 1]/2}$,
\begin{equation}\label{eq-bmcts2}
 \tau(2\epsilon^2)\le T_2(\mu,\epsilon)\le\frac{6}{\epsilon^4}\tau(\epsilon^2/2).
\end{equation}

\item[(2)] For the discrete time case, one has, for $0<c<\pi(|\mu/\pi|^2)-1$ and $A>0$,
\begin{equation}\label{eq-bmdis1}
 \frac{\alpha(c)}{\alpha(c)+A}\left(T_2\left(\mu,\sqrt{c+\frac{A+\alpha(c)}{Ae^{\alpha(c) A}}}\right)-1\right)\le\tau(c)\le T_2\left(\mu,\sqrt{\frac{c}{1+c}}\right).
\end{equation}
In particular, for $0<\epsilon<\sqrt{[\pi(|\mu/\pi-1|^2)\wedge 1]/2}$,
\begin{equation}\label{eq-bmdis2}
 \tau(2\epsilon^2)\le T_2(\mu,\epsilon)\le\frac{6}{\epsilon^4}\tau(\epsilon^2/2)+1.
\end{equation}
\end{itemize}
\end{prop}

\begin{proof}
By Proposition \ref{p-ltcomp}, (\ref{eq-bmcts1})-(\ref{eq-bmdis1}) follow immediately from (\ref{eq-mixing1}) and (\ref{eq-bmcts2})-(\ref{eq-bmdis2}) are obvious from (\ref{eq-bm0}), while $T_2(\mu,\cdot)$ is integer-valued and there is a modification of $-1$ in (\ref{eq-bmdis1}).
\end{proof}

\begin{rem}\label{r-comp}
It is easy to see from (\ref{eq-bmdis2}) that, in Theorem \ref{t-l2cutdis}, the prerequisite of $T_{n,2}(\mu_n,\epsilon_0)\ra\infty$ for some $\epsilon_0>0$ is in fact equivalent to $\tau_n(c)\ra\infty$ for some $c>0$. By (\ref{eq-bmcts2}), such an equivalence also holds in the continuous time case.
\end{rem}

\begin{rem}
Set $\theta=\inf_{n,x}K_n(x,x)$. Clearly, $(K_n-\theta I)/(1-\theta)$ is a stochastic matrix and this implies that the eigenvalues of $K_n$ fall in $[2\theta-1,1]$. Referring to the setting in (\ref{eq-jc}), if $\theta>1/2$, then $\lambda_{n,j_n(c)}\le -\log(2\theta-1)$. In this case, the right sides of (\ref{eq-max1})-(\ref{eq-max2}) turn into the same forms as in Theorem \ref{t-l2cutcts}.
\end{rem}

\subsection{Comparisons of $L^2$-cutoffs}

In the total variation, a comparison of cutoffs was made in \cite{CSal13-1} between continuous time chains and lazy discrete time chains. In this subsection, we consider the same comparison issue in the $L^2$-distance. For convenience, we shall use the following notations only in this subsection. For any discrete time Markov chain $(\mathcal{S},K,\pi)$ and $\theta\in(0,1)$, its $\theta$-lazy version refers to the discrete time chain $(\mathcal{S},K_\theta,\pi)$, where $K_\theta:=\theta I+(1-\theta)K$, and its associated continuous time chain refers to $(\mathcal{S},L,\pi)$, where $L=K-I$.

\begin{thm}\label{t-comparison}
Consider a family of irreducible and reversible discrete time finite Markov chains $\mathcal{F}=(\mu_n,\mathcal{S}_n,K_n,\pi_n)_{n=1}^\infty$. Let $\mathcal{F}_c$ and $\mathcal{F}_\theta$ with $\theta\in(0,1)$ be respective families of continuous time chains and $\theta$-lazy chains associated with $\mathcal{F}$. For $n\ge 1$, let $T_{n,2}^{(c)}(\mu_n,\cdot)$ and $T_{n,2}^{(\theta)}(\mu_n,\cdot)$ be the $L^2$-mixing times of the $n$th chains in $\mathcal{F}_c$ and $\mathcal{F}_\theta$. Assume that $\pi_n(|\mu_n/\pi_n|^2)\ra\infty$.
\begin{itemize}
\item[(1)] For $\theta\in[1/2,1)$, if $\mathcal{F}_\theta$ has a $L^2$-cutoff and $T_{n,2}^{(\theta)}(\mu_n,\epsilon_0)\ra\infty$ for some $\epsilon_0>0$, then $\mathcal{F}_c$ has a $L^2$-cutoff.

\item[(2)] If $\mathcal{F}_c$ has a $L^2$-cutoff and $T_{n,2}^{(c)}(\mu_n,\epsilon_0)\ra\infty$ for some $\epsilon_0>0$, then $\mathcal{F}_\theta$ has a $L^2$-cutoff for all $\theta\in(1/2,1)$.
\end{itemize}

In particular, for $\theta\in(1/2,1)$, if $\mathcal{F}_c$ and $\mathcal{F}_\theta$ have $L^2$-cutoffs and there is $\epsilon_0>0$ such that $T_{n,2}^{(c)}(\mu_n,\epsilon_0)\ra\infty$ or $T_{n,2}^{(\theta)}(\mu_n,\epsilon_0)\ra\infty$, then
\[
 1-\theta\le\liminf_{n\ra\infty}\frac{T_{n,2}^{(c)}(\mu_n,\epsilon)}
 {T_{n,2}^{(\theta)}(\mu_n,\epsilon)}\le\limsup_{n\ra\infty}
 \frac{T_{n,2}^{(c)}(\mu_n,\epsilon)}{T_{n,2}^{(\theta)}(\mu_n,\epsilon)}
 \le\frac{-\log(2\theta-1)}{2},\quad\forall \epsilon>0.
\]
\end{thm}

\begin{rem}\label{r-lazy}
Refer to Theorem \ref{t-comparison} and let $(\mu_n,\mathcal{S}_n,K_{n,\theta},\pi_n)$ be the $\theta$-lazy version of the $n$th chain in $\mathcal{F}$. Consider the following computations.
\begin{align}
 \pi_n\left(\left|\frac{\mu_n}{\pi_n}\right|^2\right)\ge
 \pi_n\left(\left|\frac{\mu_nK_{n,\theta}}{\pi_n}\right|^2\right)
 &=\pi_n\left(\left|\theta\frac{\mu_n}{\pi_n}+(1-\theta)\frac{\mu_nK_n}{\pi_n}
 \right|^2\right)\notag\\&\ge\theta^2\pi_n\left(\left|\frac{\mu_n}{\pi_n}\right|^2\right).\notag
\end{align}
This implies that $\pi_n(|\mu_n/\pi_n|^2)\ra\infty$ if and only if $\pi_n(|\mu_nK_{n,\theta}/\pi_n|^2)\ra\infty$ for all $\theta\in(0,1)$.
\end{rem}

\begin{rem}
In \cite{CSal13-1}, Chen and Saloff-Coste proved that, when $\mathcal{F}_c$ and $\mathcal{F}_\theta$ present cutoffs in the total variation, the ratio of their cutoff times tends to a constant dependent on $\theta$ but independent of Markov chains. In general, this observation can fail in the $L^2$-distance. To see an example, let $\pi_n$ be a probability on $\mathcal{S}_n=\{0,1,...,n\}$ and $K_n(x,y)=r\delta_x(y)+(1-r)\pi(y)$, where $r\in(0,1)$ and $\delta_x$ is the Dirac delta function. For $\theta\in(0,1)$, let $K_{n,\theta}$ be the $\theta$-lazy version of $K_n$ and $L_n=K_n-I$. It is easy to see that $1-r$ and $\theta+(1-\theta)r$ are eigenvalues of $-L_n$ and $K_{n,\theta}$ with multiplicities $n$.
Referring to the notations in (\ref{eq-jc})-(\ref{eq-tauc}), we use $j_n(c),j_{n,\theta}(c)$ and $\tau_n(c),\tau_{n,\theta}(c)$ to denote the corresponding constants associated with $L_n,K_{n,\theta}$. When $\mu_n=\delta_{x_n}$ with $x_n\in\mathcal{S}_n$ and $1/\pi_n(x_n)-1>c$, one has $j_n(c)=j_{n,\theta}(c)=1$ and
\[
 \tau_n(c)=\frac{\log(1/\pi_n(x_n)-1)}{2(1-r)},\quad
 \tau_{n,\theta}(c)=\frac{\log(1/\pi_n(x_n)-1)}{-2\log(\theta+(1-\theta)r)}.
\]
By Theorems \ref{t-l2cutcts}-\ref{t-l2cutdis}, if $\pi_n(x_n)\ra 0$, then $\mathcal{F}_c$ and $\mathcal{F}_\theta$ have $L^2$-cutoffs with cutoff times $\tau_n(c)$ and $\tau_{n,\theta}(c)$. Note that
\[
 \frac{\tau_n(c)}{\tau_{n,\theta}(c)}=\frac{-\log(\theta+(1-\theta)r)}
 {1-r},
\]
where the right side takes values on $(1-\theta,-\log\theta)$ when $r$ ranges over $(0,1)$.
\end{rem}

\begin{proof}[Proof of Theorem \ref{t-comparison}]
We first make some spectral analysis for chains in $\mathcal{F}_c$ and $\mathcal{F}_\theta$.
Let $(\mu_n,\mathcal{S}_n,K_n,\pi_n)$ be the $n$th chain in $\mathcal{F}$ and let $\beta_{n,0}=1>\beta_{n,1}\ge\cdots\ge\beta_{n,|\mathcal{S}_n|-1}$ be eigenvalues of $K_n$. Set $\lambda_{n,i}=1-\beta_{n,i}$ and $\beta_{n,i}^{(\theta)}=\theta+(1-\theta)\beta_{n,i}$. It is easy to see that, for the $n$th chains in $\mathcal{F}_c$ and $\mathcal{F}_\theta$, the infinitesimal generator and the transition matrix have eigenvalues $(-\lambda_{n,i})_{i=0}^{|\mathcal{S}_n|-1}$ and $(\beta_{n,i}^{(\theta)})_{i=0}^{|\mathcal{S}_n|-1}$ with common $L^2(\pi_n)$-orthonormal right eigenvectors. Let $j_n(c),j_{n,\theta}(c)$ and $\tau_n(c),\tau_{n,\theta}(c)$ be the constants in (\ref{eq-jc})-(\ref{eq-tauc}) for the $n$th chains in $\mathcal{F}_c,\mathcal{F}_\theta$. Note that $\beta_{n,i}^{(\theta)}\ge 2\theta-1$ for all $i\ge 1$. When $\theta\in[1/2,1)$, one has $\beta_{n,|\mathcal{S}_n|-1}^{(\theta)}\ge 0$. This implies $j_n(c)=j_{n,\theta}(c)$ for all $c>0$ and, by the following inequalities,
\begin{equation}\label{eq-log}
 \log t\le t-1,\quad\forall 0<t\le 1,\quad \log t\ge\frac{\log a}{1-a}(1-t),\quad\forall 0<a<t\le 1,
\end{equation}
we have
\begin{equation}\label{eq-betacomp}
 -\log\beta_{n,i}^{(\theta)}\begin{cases}\ge (1-\theta)\lambda_{n,i}&\text{for }\theta\in[1/2,1),\\
 \le2^{-1}[-\log(2\theta-1)]\lambda_{n,i}&\text{for }\theta\in(1/2,1),\end{cases}
\end{equation}
and, for all $c>0$,
\begin{equation}\label{eq-tauncomp}
 \tau_{n,\theta}(c)\begin{cases}\le(1-\theta)^{-1}\tau_n(c)&\text{for }\theta\in[1/2,1),\\
 \ge2[-\log(2\theta-1)]^{-1}\tau_n(c)&\text{for }\theta\in(1/2,1).\end{cases}
\end{equation}

Now, we are ready to prove this theorem. For (1), let $\theta\in[1/2,1)$ and assume that $\mathcal{F}_\theta$ has a $L^2$-cutoff with $T_{n,2}^{(\theta)}(\mu_n,\epsilon_0)\ra\infty$ for some $\epsilon_0>0$. By Remark \ref{r-lazy} and Theorem \ref{t-l2cutdis}, one has
\begin{equation}\label{eq-tauntheta}
 \tau_{n,\theta}(c)\left(-\log\beta_{n,j_{n,\theta}(c)}^{(\theta)}\right)\ra\infty,\quad
 \tau_{n,\theta}(c)\ra\infty,\quad \forall c>0.
\end{equation}
For the case $\theta\in(1/2,1)$, one may use the second inequality in (\ref{eq-betacomp}) and the first inequality in (\ref{eq-tauncomp}) to conclude $\tau_n(c)\lambda_{n,j_n(c)}\ra\infty$ for all $c>0$. By Theorem \ref{t-l2cutcts}, this implies that $\mathcal{F}_c$ has a $L^2$-cutoff. For the case $\theta=1/2$, note that if
$\beta_{n,j}^{(1/2)}\in[0,1/2]$, then $\lambda_{n,j}=2\left(1-\beta_{n,j}^{(1/2)}\right)\ge 1$. If $\beta_{n,j}^{(1/2)}\in(1/2,1)$, then the application of the second inequality in (\ref{eq-log}) with $a=1/2$ yields
$\lambda_{n,j}\ge -\log\beta_{n,j}^{(1/2)}$. As a consequence, we obtain $\lambda_{n,j}\ge\min\left\{-\log\beta_{n,j}^{(1/2)},1\right\}$. By the first inequality in (\ref{eq-tauncomp}) and (\ref{eq-tauntheta}), this leads to
\[
 \tau_n(c)\lambda_{n,j_n(c)}\ge\frac{1}{2}
 \min\left\{\tau_{n,1/2}(c)\left(-\log\beta_{n,j_n(c)}^{(1/2)}\right),\tau_{n,1/2}(c)\right\}\ra\infty,
\]
which proves that $\mathcal{F}_c$ has a $L^2$-cutoff.

For (2), assume that $\mathcal{F}_c$ has a $L^2$-cutoff and, for some $\epsilon_0>0$, $T_{n,2}^{(c)}(\mu_n,\epsilon_0)\ra\infty$. By Theorem \ref{t-l2cutcts}, $\tau_n(c)\lambda_{n,j_n(c)}\ra\infty$ and $\tau_n(c)\ra\infty$ for all $c>0$. Combining the first inequality in (\ref{eq-betacomp}) and the second inequality in (\ref{eq-tauncomp}), we obtain (\ref{eq-tauntheta}) for $\theta\in(1/2,1)$ and, by Theorem \ref{t-l2cutdis}, $\mathcal{F}_\theta$ has a $L^2$-cutoff. The comparison of the $L^2$-cutoff times is immediate from (\ref{eq-tauncomp}).
\end{proof}

\begin{rem}\label{r-comp2}
From the proof of Theorem \ref{t-comparison}, we would like to remark the observation that, for $\theta\in(1/2,1)$, $T_{n,2}^{(\theta)}(\mu_n,\epsilon)\ra\infty$ for some $\epsilon>0$ if and only if $T_{n,2}^{(c)}(\mu_n,\epsilon)\ra\infty$ for some $\epsilon>0$. Note that this can also be proved using Proposition \ref{p-bm} and (\ref{eq-tauncomp}).
\end{rem}

In the following corollary, the laziness is combined with $\mathcal{F}$ and the comparison of cutoffs between $\mathcal{F}$ and $\mathcal{F}_c$ is summarized from Theorem \ref{t-comparison}.

\begin{cor}\label{c-comparison}
Let $\mathcal{F}$ be a family of irreducible and reversible discrete time finite Markov chain and $\mathcal{F}_c$ be the family of continuous time chains associated with $\mathcal{F}$. Assume that $\inf_{n,x}K_n(x,x)>1/2$, $\pi_n(|\mu_n/\pi_n|^2)\ra\infty$ and there is $\epsilon_0>0$ such that $T_{n,2}(\mu_n,\epsilon_0)\ra 0$ or $T_{n,2}^{(c)}(\mu_n,\epsilon_0)\ra\infty$. Then, $\mathcal{F}$ has a $L^2$-cutoff if and only if $\mathcal{F}_c$ has a $L^2$-cutoff.
\end{cor}
\begin{proof}
Set $\theta=\inf_{n,x}K_n(x,x)$ and $\widetilde{K}_n=(K_n-\theta I)/(1-\theta)$. The proof follows immediately from the observation of $K_n=\theta I+(1-\theta)\widetilde{K}_n$ and $e^{t(K_n-I)}=e^{(1-\theta)t(\widetilde{K}_n-I)}$, and the application of Theorem \ref{t-comparison} and Remark \ref{r-comp2} to the family of $(\mu_n,\mathcal{S}_n,\widetilde{K}_n,\pi_n)_{n=1}^\infty$.
\end{proof}

\section{Products chains}\label{s-cexp}

In this section, we consider families of continuous time product chains. Let
\begin{equation}\label{eq-f}
 \mathcal{F}=\{(\mu_{n,i},\mathcal{S}_{n,i},L_{n,i},\pi_{n,i})|1\le i\le \ell_n,n\ge 1\}
\end{equation}
be a triangular array of irreducible continuous time finite Markov chains and \begin{equation}\label{eq-p}
 \mathcal{P}=\{p_{n,i}|1\le i\le \ell_n,n\ge 1\}
\end{equation}
be a triangular array of positive reals satisfying $p_{n,1}+\cdots+p_{n,\ell_n}\le 1$. For $n\ge 1$, set $\mathcal{S}_n=\mathcal{S}_{n,1}\times\cdots\times \mathcal{S}_{n,\ell_n}$, $\mu_n=\mu_{n,1}\times\dots\times\mu_{n,\ell_n}$, $\pi_n=\pi_{n,1}\times\cdots\times\pi_{n,\ell_n}$ and define
\begin{equation}\label{eq-ln}
 L_n=\sum_{i=1}^{\ell_n}p_{n,i}I_{n,1}\otimes\cdots\otimes I_{n,i-1}\otimes L_{n,i}\otimes I_{n,i+1}\otimes\cdots\otimes I_{n,\ell_n},
\end{equation}
where $I_{n,i}$ is the identity matrix indexed by $\mathcal{S}_{n,i}$ and $M\otimes M'$ denotes the tensor product of matrices $M$ and $M'$. In what follows, we write $\mathcal{F}^\mathcal{P}$ for $(\mu_n,\mathcal{S}_n,L_n,\pi_n)_{n=1}^\infty$ and call it the family of product chains induced by $\mathcal{F}$ and $\mathcal{P}$.

\subsection{The $L^2$-cutoffs of product chains}

Referring to the setting in (\ref{eq-ln}), if $H_{n,i,t}=e^{tL_{n,i}}$ and $H_{n,t}=e^{tL_n}$, then
\begin{equation}\label{eq-prodcts}
 H_{n,t}=H_{n,1,p_1t}\otimes\cdots\otimes H_{n,\ell_n,p_{\ell_n}t}.
\end{equation}
This leads to the following proposition.

\begin{prop}\label{p-prod}
Let $\mathcal{F},\mathcal{P}$ be as in \textnormal{(\ref{eq-f})-(\ref{eq-p})} and $\mathcal{F}^\mathcal{P}$ be the family of product chains induced by $\mathcal{F}$ and $\mathcal{P}$. For $n\ge 1$ and $1\le i\le \ell_n$, let $d_{n,2}(\mu_n,\cdot)$ and $d_{n,i,2}(\mu_{n,i},\cdot)$ be the $L^2$-distances of $(\mu_n,\mathcal{S}_n,L_n,\pi_n)$ and $(\mu_{n,i},\mathcal{S}_{n,i},L_{n,i},\pi_{n,i})$. Then, $\mathcal{F}^\mathcal{P}$ has a $L^2$-cutoff if and only if there is a sequence of positive reals $(t_n)_{n=1}^\infty$ such that
\[
 \lim_{n\ra\infty}\sum_{i=1}^{\ell_n}d_{n,i,2}(\mu_{n,i},ap_{n,i}t_n)^2=\begin{cases}
 0&\text{for }a>1,\\\infty&\text{for }0<a<1.\end{cases}
\]
Further, if $T_{n,2}(\mu_n,\cdot)$ is the $L^2$-mixing time of $(\mu_n,\mathcal{S}_n,L_n,\pi_n)$ and
\[
 \mathcal{T}_n(\epsilon)=\min\left\{t\ge 0\bigg|\sum_{i=1}^{\ell_n}d_{n,i,2}(\mu_{n,i},p_{n,i}t)^2\le\epsilon\right\},
\]
then
\begin{equation}\label{eq-prodmixing}
 T_{n,2}(\mu_n,\sqrt{e^\epsilon-1})\le \mathcal{T}_n(\epsilon)\le T_{n,2}(\mu_n,\sqrt{\epsilon}).
\end{equation}
\end{prop}
\begin{proof}
By (\ref{eq-prodcts}), one has
\begin{equation}\label{eq-l2product}
 d_{n,2}(\mu_n,t)^2=\prod_{i=1}^{\ell_n}\left(d_{n,i,2}(\mu_{n,i},p_{n,i}t)^2+1\right)
 -1.
\end{equation}
This implies
\[
 \sum_{i=1}^{\ell_n}d_{n,i,2}(\mu_{n,i},p_{n,i}t)^2\le d_{n,2}(\mu_n,t)^2
 \le\exp\left\{\sum_{i=1}^{\ell_n}d_{n,i,2}(\mu_{n,i},p_{n,i}t)^2\right\}-1.
\]
The remaining of the proof follows from the above inequalities.
\end{proof}

\begin{rem}\label{r-prod1}
In general, the identity in (\ref{eq-l2product}) does not hold in the discrete time case. To see the details, let $\mathcal{F}=\{(\mu_{n,i},\mathcal{S}_{n,i},K_{n,i},\pi_{n,i})|1\le i\le \ell_n,n\ge 1\}$, $\mathcal{P}$ be as in (\ref{eq-p}) and $\mathcal{F}^{\mathcal{P}}=(\mu_n,\mathcal{S}_n,K_n,\pi_n)_{n=1}^\infty$, where
\[
  K_n=p_{n,0}I+\sum_{i=1}^{\ell_n}p_{n,i}I_{n,1}\otimes\cdots\otimes I_{n,i-1}\otimes K_{n,i}\otimes I_{n,i+1}\otimes\cdots\otimes I_{n,\ell_n},
\]
and $p_{n,0}=1-(p_{n,1}+\cdots+p_{n,\ell_n})$. For simplicity, we assume that $K_{n,i}$ is reversible and let $\{\beta_{n,i,j}|0\le j<|\mathcal{S}_{n,i}|\}$ and $\{\phi_{n,i,j}|0\le j<|\mathcal{S}_{n,i}|\}$ be eigenvalues and $L^2(\pi_{n,i})$-orthonormal right eigenvectors of $K_{n,i}$. For $J=(j_1,...,j_{\ell_n})$ with $0\le j_i<|\mathcal{S}_{n,i}|$ and $1\le i\le \ell_n$, set $\beta_{n,J}=p_{n,0}+\sum_{i=1}^{\ell_n}p_{n,i}\beta_{n,i,j_i}$ and $\phi_{n,J}=\phi_{n,1,j_1}\otimes\cdots\otimes\phi_{n,\ell_n,j_{\ell_n}}$. It is easy to see that $\beta_{n,J}$'s are eigenvalues of $K_n$ with $L^2(\pi_n)$-orthonormal right eigenvectors $\phi_{n,J}$'s. As a consequence, if $\beta_{n,i,0}=1$, then the $L^2$-distance, $d_{n,2}(\mu_n,\cdot)$, of $(\mu_n,\mathcal{S}_n,K_n,\pi_n)$ satisfies
\[
 d_{n,2}(\mu_n,m)^2=\sum_{J:J\ne\mathbf{0}}|\mu_n(\phi_{n,J})|^2\beta_{n,J}^{2m},
\]
where $\mathbf{0}=(0,0,...,0)$ and $\mu_n(\phi_{n,J})=\prod_{i=1}^{\ell_n}\mu_{n,i}(\phi_{n,i,j_i})$. In the continuous time case of (\ref{eq-f})-(\ref{eq-ln}), if $\{\lambda_{n,i,j}|0\le j<|\mathcal{S}_{n,i}|\}$ are eigenvalues of $L_{n,i}$ with $L^2(\pi_{n,i})$-orthonormal right eigenvectors $\{\phi_{n,i,j}|0\le j<|\mathcal{S}_{n,i}|\}$, then $\lambda_{n,J}=\sum_{i=1}^{\ell_n}p_{n,i}\lambda_{n,i,j_i}$ is an eigenvalue of $-L_n$ with right eigenvector $\phi_{n,J}$ defined as before. When $\lambda_{n,i,0}=0$, this implies
\[
 d_{n,2}(\mu_n,t)^2=\sum_{J:J\ne\mathbf{0}}|\mu_n(\phi_{n,J})|^2e^{-t\lambda_{n,J}},
\]
which is exactly the formula in (\ref{eq-l2product}). It is worth while to note that, in Proposition \ref{p-prod}, the reversibility is not required.
\end{rem}

\begin{thm}\label{t-prod}
Let $\mathcal{F},\mathcal{P}$ be the triangular arrays in \textnormal{(\ref{eq-f})-(\ref{eq-p})}. Assume that chains in $\mathcal{F}$ are reversible and let $\lambda_{n,i,0}=0$, $\lambda_{n,i,1}$,...,$\lambda_{n,i,|\mathcal{S}_{n,i}|-1}$ be eigenvalues of $-L_{n,i}$ with $L^2(\pi_{n,i})$-orthonormal right eigenvectors $\phi_{n,i,0}=\mathbf{1}$, $\phi_{n,i,1}$,...,$\phi_{n,i,|\mathcal{S}_{n,i}|-1}$. Set
\[
 \left\{\rho_{n,l}\bigg|1\le l\le \sum_{i=1}^{\ell_n}|\mathcal{S}_{n,i}|-\ell_n\right\}=\{p_{n,i}\lambda_{n,i,j}|1\le j<|\mathcal{S}_{n,i}|,1\le i\le \ell_n\}
\]
in the way that $\rho_{n,l}\le \rho_{n,l+1}$ and arrange accordingly
\[
 \left\{\psi_{n,l}\bigg|1\le l\le \sum_{i=1}^{\ell_n}|\mathcal{S}_{n,i}|-\ell_n\right\}=\{\mu_{n,i}(\phi_{n,i,j})|1\le j<|\mathcal{S}_{n,i}|,1\le i\le \ell_n\}.
\]
Let $T_{n,2}(\mu_n,\cdot)$ be the $L^2$-mixing time of the $n$th chain in $\mathcal{F}^\mathcal{P}$ and, for $c>0$, define
\begin{equation}\label{eq-jnprod}
 \widetilde{j}_n(c)=\min\left\{j\ge 1\bigg|\sum_{l=1}^j\psi_{n,l}^2>c\right\}
\end{equation}
and
\begin{equation}\label{eq-taunprod}
 \widetilde{\tau}_n(c)=\max_{j\ge j_n(c)}\left\{\frac{\log\left(1+\sum_{l=1}^j|\psi_{n,l}|^2\right)}{2\rho_{n,j}}\right\}
\end{equation}
Then, $\mathcal{F}^\mathcal{P}$ has a $L^2$-cutoff if and only if $\widetilde{\tau}_n(c)\rho_{n,\widetilde{j}_n(c)}\ra\infty$ for all $c>0$. Further, if $\mathcal{F}^\mathcal{P}$ has a $L^2$-cutoff, then $\widetilde{\tau}_n(c)$ is a cutoff time and, for all $\epsilon>0$ and $c>0$,
\begin{equation}\label{eq-prodmixing2}
 |T_{n,2}(\mu_n,\epsilon)-\widetilde{\tau}_n(c)|
 =O\left(\sqrt{\widetilde{\tau}_n(c)/\rho_{n,\widetilde{j}_n(c)}}\right).
\end{equation}
\end{thm}

\begin{proof}
Let $\mathcal{T}_n$ be as in Proposition \ref{p-prod} and set
\[
 f_n(t):=\sum_{l\ge 1}\psi_{n,l}^2e^{-2\rho_{n,l}t}=\sum_{i=1}^{\ell_n}d_{n,i,2}(\mu_{n,i},p_{n,i}t)^2.
\]
By Proposition \ref{p-prod}, $\mathcal{F}^{\mathcal{P}}$ has a $L^2$-cutoff if and only if $(f_n)_{n=1}^\infty$ has a cutoff. Note that $f_n$ can be regarded as a Laplace transform of some discrete measure on $[0,\infty)$. By Theorem \ref{t-ltcutoff}, $(f_n)_{n=1}^\infty$ has a cutoff if and only if $\widetilde{\tau}_n(c)\rho_{n,\widetilde{j}_n(c)}\ra\infty$ for all $c>0$. Further, as a consequence of
(\ref{eq-mixbound2}), if $(f_n)_{n=1}^\infty$ has a cutoff, then
\[
 |\mathcal{T}_n(\epsilon)-\widetilde{\tau}_n(c)|
 =O\left(\sqrt{\widetilde{\tau}_n(c)/\rho_{n,\widetilde{j}_n(c)}}\right),
 \quad\forall c,\epsilon\in(0,\infty).
\]
The desired comparison in (\ref{eq-prodmixing2}) is then given by the above identity and (\ref{eq-prodmixing}).
\end{proof}

\begin{rem}\label{r-prod2}
Note that $\widetilde{j}_n,\widetilde{\tau}_n$ in Theorem \ref{t-prod} are different from $j_n,\tau_n$ in Theorem \ref{t-l2cutcts}, while Lemma \ref{l-jtaucomp} provides a comparison between each other, which is crucial for the discussion in Example \ref{ex-counterexample}.
\end{rem}

\subsection{Products of two-state chains}

In this subsection, we restrict ourselves to products of two-state chains and derive a simplified method to determine cutoffs from Theorem \ref{t-prod}. For convenience, we shall restrict ourselves to the continuous time case and all chains in $\mathcal{F}^\mathcal{P}$ will be assumed to start at $\mathbf{0}$, the zero vector.

\begin{thm}\label{t-prod2state}
Let $\mathcal{F},\mathcal{P}$ be triangular arrays in \textnormal{(\ref{eq-f})-(\ref{eq-p})} with $\mathcal{S}_{n,i}=\{0,1\}$, $\mu_{n,i}=\delta_0$ and
\[
 L_{n,i}=\left(\begin{array}{cc}-A_{n,i}&A_{n,i}\\B_{n,i}&-B_{n,i}
 \end{array}\right).
\]
For $n\ge 1$, let $T_{n,2}(\mathbf{0},\cdot)$ be the $L^2$-mixing time of the $n$th chain in $\mathcal{F}^\mathcal{P}$. Suppose that $p_{n,i}\le p_{n,i+1}$ for $1\le i<\ell_n$ and there are a constant $R>1$ and a sequence of positive reals $r_n$ such that
\begin{equation}\label{eq-arn}
 R^{-1}r_n\le A_{n,i}\le Rr_n,\quad R^{-1}r_n\le B_{n,i}\le Rr_n,\quad\forall 1\le i\le \ell_n,\,n\ge 1.
\end{equation}
Then, $\mathcal{F}^{\mathcal{P}}$ has a $L^2$-cutoff if and only if
\begin{equation}\label{eq-2statel2cut}
 \lim_{n\ra\infty}\max_{j\ge 1}\frac{\log (1+j)}{p_{n,j}/p_{n,1}}=\infty.
\end{equation}
Moreover, assuming that $p_{n,i}(A_{n,i}+B_{n,i})$ is increasing in $i$ for all $n\ge 1$, one has
\begin{equation}\label{eq-nocutmixing}
 R^{-2}t_n\le T_{n,2}(\mathbf{0},\epsilon)\le 40R^2\epsilon^{-4}t_n,\quad
 \forall 0<\epsilon<1/(\sqrt{2}R),
\end{equation}
and, further, if \textnormal{(\ref{eq-2statel2cut})} holds, then
\[
 T_{n,2}(\mathbf{0},\epsilon)=t_n+O(b_n),\quad\forall \epsilon>0,
\]
where
\begin{equation}\label{eq-tnbn}
 t_n=\max_{j\ge 1}\frac{\log(1+j)}{2p_{n,j}(A_{n,j}+B_{n,j})},\quad b_n=\sqrt{\frac{t_n}{r_np_{n,1}}}.
\end{equation}
\end{thm}

\begin{proof}
Note that $-(A_{n,i}+B_{n,i})$ is the non-zero eigenvalue of $L_{n,i}$ with $L^2(\pi_{n,i})$-orthonormal right eigenvector $\phi_{n,i}=(\sqrt{A_{n,i}/B_{n,i}},\sqrt{B_{n,i}/A_{n,i}})$. Let $\rho_{n,i}$ be an increasing arrangement of $p_{n,i}(A_{n,i}+B_{n,i})$ and $\psi_{n,i}$ be an arrangement of $\sqrt{A_{n,i}/B_{n,i}}$ accordingly. For $c>0$, let $\widetilde{j}_n(c),\widetilde{\tau}_n(c)$ be constants defined in (\ref{eq-jnprod})-(\ref{eq-taunprod}). By Theorem \ref{t-prod}, $\mathcal{F}^\mathcal{P}$ has a $L^2$-cutoff if and only if $\widetilde{\tau}_n(c)\rho_{n,\widetilde{j}_n(c)}\ra\infty$ for all $c>0$.

Based on the assumption of (\ref{eq-arn}), it is easy to see that
\begin{equation}\label{eq-aa}
 R^{-1}\le\psi_{n,i}\le R,\quad 2R^{-1}r_np_{n,i}\le\rho_{n,i}\le 2Rr_np_{n,i},
 \quad\forall 1\le i\le \ell_n.
\end{equation}
Using the following inequalities.
\[
  \forall t>0,\quad \log(1+at)-a\log(1+t)\begin{cases}<0&\text{for }a>1,\\>0&\text{for }0<a<1,\end{cases}
\]
one may derive from (\ref{eq-aa}) that
\begin{equation}\label{eq-jnc}
 R^{-2}\log(1+j)\le\log\left(1+\sum_{i=1}^j\psi_{n,i}^2\right)\le R^2\log(1+j),
\end{equation}
and then
\[
 \frac{1}{4R^3r_n}s_n(\widetilde{j}_n(c))\le\widetilde{\tau}_n(c)\le \frac{R^3}{4r_n}s_n(\widetilde{j}_n(c)),
\]
where
\[
 s_n(l)=\max_{j\ge l}\left\{\frac{\log(1+j)}{p_{n,j}}\right\}.
\]
As a consequence, $\mathcal{F}^{\mathcal{P}}$ has a $L^2$-cutoff if and only if $p_{n,\widetilde{j}_n(c)}s_n(\widetilde{j}_n(c))\ra\infty$ for all $c>0$.

Let $R$ be the constant as before. By the first inequality of (\ref{eq-aa}), one has $j_n(c)=1$ for all $0<c<R^{-2}$ and $n\ge 1$. This implies that if $\mathcal{F}^{\mathcal{P}}$ has a $L^2$-cutoff, then $p_{n,1}s_n(1)\ra\infty$. Conversely, we assume that $p_{n,1}s_n(1)\ra\infty$. Note that
\[
 p_{n,i}s_n(i)\le \max\{\log j,p_{n,j}s_n(j)\},\quad\forall i\le j.
\]
As a result, this implies $p_{n,j}s_n(j)\ra\infty$ for all $j\ge 1$. Following (\ref{eq-jnc}), we obtain that, for any $c>0$, $\widetilde{j}_n(c)$ is bounded and this leads to $p_{n,\widetilde{j}_n(c)}s_n(\widetilde{j}_n(c))\ra\infty$, which proves the equivalence of the $L^2$-cutoff of $\mathcal{F}^\mathcal{P}$.

To bound the $L^2$-mixing time, we assume that $p_{n,i}(A_{n,i}+B_{n,i})$ is increasing in $i$ for all $n\ge 1$. In this case, $\rho_{n,i}=p_{n,i}(A_{n,i}+B_{n,i})$ and, by (\ref{eq-jnc}), one has
\begin{equation}\label{eq-nocutmixing1}
 \widetilde{j}_n(c)=1,\quad R^{-2}t_n\le\widetilde{\tau}_n(c)\le R^2t_n,\quad\forall c\in(0,R^{-2}),
\end{equation}
where $t_n$ is the constant in (\ref{eq-tnbn}). Let $\mathcal{T}_n(\epsilon)$ be the corresponding constant in Proposition \ref{p-prod}. As a result, we have
\begin{equation}\label{eq-nocutmixing2}
 \mathcal{T}_n(\epsilon^2)\le T_{n,2}(\mathbf{0},\epsilon)\le\mathcal{T}_n(\log(1+\epsilon^2))
 \le\mathcal{T}_n(\epsilon^2/2),\quad\forall \epsilon\in(0,1),
\end{equation}
where the last inequality uses the fact of $\log(1+t)\ge t/(1+t)$ for all $t\ge 0$. By Proposition \ref{p-ltcomp}, (\ref{eq-bm0}) yields
\begin{equation}\label{eq-nocutmixing3}
 \widetilde{\tau}_n(2\delta)\le\mathcal{T}_n(\delta)
 \le\frac{12}{\delta^2}\widetilde{\tau}_n(\delta/2),\quad\forall \delta\in(0,1/(2R^2)).
\end{equation}
Consequently, (\ref{eq-nocutmixing}) follows immediately from (\ref{eq-nocutmixing1}), (\ref{eq-nocutmixing2}) and (\ref{eq-nocutmixing3}).

To estimate the $L^2$-cutoff time, we assume that (\ref{eq-2statel2cut}) holds. Let $t_n,b_n$ be those constants in (\ref{eq-tnbn}) and $c\in(0,R^{-2})$. As before, we have $\widetilde{j}_n(c)=1$ for all $n\ge 1$ and
\[
 \widetilde{\tau}_n(c)=\max_{j\ge 1}\frac{\log(1+\sum_{l=1}^j|\psi_{n,l}|^2)}{2p_{n,j}(A_{n,j}+B_{n,j})}.
\]
By (\ref{eq-aa}), one may derive
\[
 \log(1+j)-2\log R\le\log\left(1+\sum_{l=1}^j|\psi_{n,l}|^2\right)\le\log(1+j)+2\log R,
\]
and, as a result of (\ref{eq-arn}), this yields
\begin{equation}\label{eq-tauntn}
 |\widetilde{\tau}_n(c)-t_n|\le\frac{\log R}{p_{n,1}(A_{n,1}+B_{n,1})}\le\frac{R\log R}{2r_np_{n,1}}.
\end{equation}
It is easy to check, using (\ref{eq-2statel2cut}), that $(r_np_{n,1})^{-1}=o(b_n)$ and $b_n=o(t_n)$. Consequently, (\ref{eq-tauntn}) leads to $\widetilde{\tau}_n(c)\sim t_n$ and, hence,
\[
 |\widetilde{\tau}_n(c)-t_n|=o(b_n),\quad
 \widetilde{\tau}_n(c)/\rho_{n,\widetilde{j}_n(c)}\asymp b_n^2.
\]
The desired identity for the $L^2$-mixing time is then given by (\ref{eq-prodmixing2}).
\end{proof}

In the next theorem, we consider specific triangular arrays $\mathcal{P}$.

\begin{thm}\label{t-prod2state2}
Let $\mathcal{F}$ be a triangular array in \textnormal{(\ref{eq-f})} with
\[
 \mathcal{S}_{n,i}=\{0,1\},\quad
 L_{n,i}=\left(\begin{array}{cc}-A_{n,i}&A_{n,i}\\B_{n,i}&-B_{n,i}
 \end{array}\right),
\]
and assume
\[
 A_{n,i}+B_{n,i}=A_{n,1}+B_{n,1},\quad\forall i\ge 1,\quad 0<\inf_{i,n}\frac{A_{n,i}}{B_{n,i}}
 \le\sup_{i,n}\frac{A_{n,i}}{B_{n,i}}<\infty.
\]
Consider a sequence of positive integers $(x_n)_{n=1}^\infty$ and a positive function $f$ defined on $(0,\infty)$. Let $\mathcal{P}$ be a triangular array in \textnormal{(\ref{eq-p})} given by
\[
 p_{n,i}=\frac{p_{n,1}f(x_n+i-1)}{f(x_n)},\quad p_{n,1}\le \frac{f(x_n)}{\sum_{i=1}^{\ell_n}f(x_n+i-1)}.
\]
\begin{itemize}
\item[(1)] If $f(t)=e^{at}$ with $a>0$, then $\mathcal{F}^\mathcal{P}$ has no $L^2$-cutoff.

\item[(2)] If $f(t)=\exp\{a[\log(1+t)]^b\}$ with $a>0$ and $b>0$, then
\[
 \mathcal{F}^\mathcal{P}\text{ has a $L^2$-cutoff}\quad\Lra\quad x_n\wedge\ell_n\ra\infty.
\]
Further, if $x_n\wedge\ell_n\ra\infty$, then
\begin{equation}\label{eq-cutofftime}
 T_{n,2}(\mathbf{0},\epsilon)=\frac{\kappa_n}{2(A_{n,1}+B_{n,1})p_{n,1}}+
 O\left(\frac{\sqrt{\kappa_n}}{(A_{n,1}+B_{n,1})p_{n,1}}\right),
 \quad\forall\epsilon>0,
\end{equation}
where $\kappa_n=(\log x_n-b\log\log x_n)\wedge\log\ell_n$.

\item[(3)] If $f(t)=[\log(1+t)]^a$ with $a>0$, then
\[
 \mathcal{F}^\mathcal{P}\text{ has a $L^2$-cutoff}\quad\Lra\quad \begin{cases}x_n\wedge\ell_n\ra\infty&\text{for }a\ge 1,\\\ell_n\ra\infty&\text{for }0<a<1.\end{cases}
\]
Further, if $a\ge 1$ and $x_n\wedge\ell_n\ra\infty$, then \textnormal{(\ref{eq-cutofftime})} holds with $\kappa_n=(\log x_n)\wedge(\log\ell_n)$. If $0<a<1$ and $\ell_n\ra\infty$, then \textnormal{(\ref{eq-cutofftime})} holds with $\kappa_n=[\log(1+x_n\wedge\ell_n)]^a(\log \ell_n)^{1-a}$.
\end{itemize}

Moreover, for Case (1), for Case (2) with $x_n\wedge\ell_n=O(1)$ and for Case (3) with $x_n\wedge\ell_n=O(1)$, when $a\ge 1$, and $\ell_n=O(1)$, when $0<a<1$, one has
\[
 T_{n,2}(\mathbf{0},\epsilon)\asymp \frac{1}{(A_{n,1}+B_{n,1})p_{n,1}},\quad\forall \epsilon\in(0,S/\sqrt{2}),
\]
where $S=\inf_{n,i}\{(A_{n,i}\wedge B_{n,i})/(A_{n,1}+B_{n,1})\}$.
\end{thm}

\begin{proof}
For $n\ge 1$, set $r_n=A_{n,1}+B_{n,1}$ and define
\[
 \Delta_n=\max_{1\le j\le \ell_n}\frac{\log(1+j)}{f(x_n-1+j)/f(x_n)}.
\]
Immediately, one can see that $0<\inf_{i,n}A_{n,i}/r_n\le\sup_{i,n}A_{n,i}/r_n<1$, which is equivalent to (\ref{eq-arn}), and, by Theorem \ref{t-prod2state}, (\ref{eq-2statel2cut}) yields
\begin{equation}\label{eq-2statel2cut2}
 \mathcal{F}^\mathcal{P}\text{ presents a $L^2$-cutoff}\quad\Lra\quad \Delta_n\ra\infty.
\end{equation}
Further, if $\Delta_n\ra\infty$, then (\ref{eq-tnbn}) implies
\begin{equation}\label{eq-2statel2cut3}
 T_{n,2}(\mathbf{0},\epsilon)=t_n+O(b_n),\quad\forall \epsilon>0,
\end{equation}
where
\[
 t_n=\frac{\Delta_n}{2r_np_{n,1}},\quad b_n=\frac{\sqrt{\Delta_n}}{r_np_{n,1}}.
\]
In what follows, we treat $f$ case by case.

For (1), assume that $f(t)=e^{at}$ with $a>0$. In this case, it is easy to see that
\[
 \Delta_n=\max_{1\le j\le \ell_n}\frac{\log(1+j)}{e^{j-1}}=\log 2,
\]
where the last inequality uses the fact that
\[
 \frac{\log(1+j)}{\log j}\le 1+\frac{1}{j\log j}<2,\quad\forall j\ge 2.
\]
As a result, $\mathcal{F}^\mathcal{P}$ has no $L^2$-cutoff for all sequences $x_n$ and $\ell_n$.

For (2), let $f(t)=\exp\{a[\log(1+t)]^b\}$ with $a>0$ and $b>0$. In this case, we define $F_c(t)=\log(1+t)/f(c-1+t)$ for $c\ge 1$ and write $\Delta_n=f(x_n)\max_{1\le j\le \ell_n}F_{x_n}(j)$. In some computations, one can show that
\[
 G_c(t):=(1+t)f(c-1+t)F_c'(t)=1-ab[\log(c+t)]^bg_c(t),\quad\forall t>0,
\]
where
\begin{equation}\label{eq-gc}
 g_c(t)=\frac{(1+t)\log(1+t)}{(c+t)\log(c+t)}.
\end{equation}
Note that the mapping $s\mapsto s\log s$ is strictly increasing on $[e^{-1},\infty)$. This implies $g_c'(t)>0$ for $t>0$ and, hence, $G_c$ is strictly decreasing on $(0,\infty)$. Along with the observation of
\[
 \lim_{t>0,t\ra 0}G_c(t)=1,\quad \lim_{t\ra\infty}G_c(t)=-\infty,\quad\forall c\ge 1,
\]
one may select, for each $c\ge 1$, a constant $t_c\in(0,\infty)$ such that $F_c'(t)>0$ for $t\in(0,t_c)$ and $F_c'(t)<0$ for $t\in(t_c,\infty)$. Consequently, this implies
\[
 \Delta_n=f(x_n)\max\{F_{x_n}((u_n-1)\vee 1),F_{x_n}(u_n)\},
\]
where $u_n=\lceil t_{x_n}\rceil\wedge\ell_n$.

Note that if $x_n\wedge \ell_n$ is bounded, then $u_n=O(1)$ and, hence, $\Delta_n\le \log(1+u_n)=O(1)$, which implies that $\mathcal{F}^\mathcal{P}$ has no $L^2$-cutoff.
Next, we assume that $x_n\wedge\ell_n\ra\infty$. In this setting, one has
\[
 \lim_{c\ra\infty}G_c\left(\frac{Ac}{(\log c)^b}\right)=1-abA,\quad\forall A>0.
\]
Clearly, this implies $t_c\sim(ab)^{-1}c(\log c)^{-b}$ as $c\ra\infty$ and, thus, we have $u_n\sim ((ab)^{-1}x_n(\log x_n)^{-b})\wedge \ell_n=o(x_n)$. To estimate $\Delta_n$, we write
\[
 [\log(1+x_n)]^b=(\log x_n)^b+by_n,\quad [\log(x_n+u_n)]^b=(\log x_n)^b+bz_n,
\]
and
\[
 \log(1+u_n)=\log u_n+v_n,\quad \frac{f(x_n)}{f(x_n-1+u_n)}=e^{ab(y_n-z_n)}=1-abw_n.
\]
It is an easy exercise to derive from the above setting that
\[
 y_n\sim\frac{(\log x_n)^{b-1}}{x_n},\quad z_n\sim\frac{u_n(\log x_n)^{b-1}}{x_n}=O\left(\frac{1}{\log x_n}\right),\quad w_n\sim z_n,\quad v_n\sim\frac{1}{u_n}.
\]
As a consequence, this leads to
\[
 \Delta_n=(1-abw_n)(\log u_n+v_n)=\log u_n+O(1)=\xi_n+O(1),
\]
where $\xi_n=(\log x_n-b\log\log x_n)\wedge\log\ell_n$. By (\ref{eq-2statel2cut2})-(\ref{eq-2statel2cut3}), $\mathcal{F}^\mathcal{P}$ has a $L^2$-cutoff and
\[
 T_{n,2}(\mathbf{0},\epsilon)=\frac{\xi_n}{2r_np_{n,1}}+O\left(\frac{\sqrt{(\log x_n)\wedge(\log\ell_n)}}{r_np_{n,1}}\right),\quad\forall \epsilon>0.
\]

For (3), we assume that $f(t)=[\log(1+t)]^a$ with $a>0$. As before, we set
\[
 \widetilde{F}_c(t)=\frac{\log(1+t)}{f(c-1+t)},\quad \widetilde{G}_c(t)=(1+t)f(c-1+t)\widetilde{F}_c'(t).
\]
Clearly, $\Delta_n=f(x_n)\max_{1\le j\le \ell_n}\widetilde{F}_{x_n}(j)$. In a similar computation, one can show that
\[
 \widetilde{G}_c(t)=1-ag_c(t),\quad \forall t>0.
\]
where $g_c$ is the function in (\ref{eq-gc}). As $g_c'>0$ on $(0,\infty)$, one may conclude that $\widetilde{G}_c$ is strictly decreasing on $(0,\infty)$. Based on the following observation
\[
 \lim_{t>0,t\ra 0}\widetilde{G}_c(t)=1,\quad \lim_{t\ra\infty}\widetilde{G}_c(t)=1-a,
\]
we treat two subcases.

{\bf Case 1: $a>1$.} Clearly, there is $\widetilde{t}_c\in(0,\infty)$ such that $\widetilde{F}_c'>0$ on $(0,\widetilde{t}_c)$ and $\widetilde{F}_c'<0$ on $(\widetilde{t}_c,\infty)$. This implies
\[
 \Delta_n=f(x_n)\max\{\widetilde{F}_{x_n}((\widetilde{u}_n-1)\vee 1),
 \widetilde{F}_{x_n}(\widetilde{u}_n)\}
\]
where $\widetilde{u}_n=\lceil \widetilde{t}_{x_n}\rceil\wedge \ell_n$. Based on the observation of
\[
 \lim_{c\ra\infty}\widetilde{G}_c(Ac)=1-\frac{aA}{A+1},\quad\forall A>0,
\]
one has $t_c\sim c/(a-1)$ as $c\ra\infty$. As a result, there exists $M>0$ such that $\widetilde{u}_n\le M(x_n\wedge\ell_n)$ for $n\ge 1$, which leads to
\[
 \Delta_n\le\log(1+\widetilde{u}_n)\le \log(1+M(x_n\wedge\ell_n)).
\]
By (\ref{eq-2statel2cut2}), if $\liminf_n x_n\wedge\ell_n<\infty$, then $\mathcal{F}^{\mathcal{P}}$ has no $L^2$-cutoff.

Next, assume that $x_n\wedge\ell_n\ra\infty$. In this case, $\widetilde{u}_n\sim(x_n/(a-1))\wedge\ell_n$ and a similar reasoning as in Case 2-2 yields
\[
 \log(1+x_n)=\log x_n+O(1/x_n), \quad \log(x_n+\widetilde{u}_n)=\log x_n+O(1),
\]
and
\[
 \log(1+\widetilde{u}_n)=\log \widetilde{u}_n+O(1/\widetilde{u}_n)=(\log x_n)\wedge (\log \ell_n)+O(1).
\]
This leads to $\Delta_n=(\log x_n)\wedge(\log \ell_n)+O(1)$. By (\ref{eq-2statel2cut3}), $\mathcal{F}^\mathcal{P}$ has a $L^2$-cutoff and
\[
 T_{n,2}(\mathbf{0},\epsilon)=\frac{(\log x_n)\wedge(\log \ell_n)}{2r_np_{n,1}}+O\left(\frac{\sqrt{(\log x_n)\wedge(\log \ell_n)}}{r_np_{n,1}}\right),\quad\forall \epsilon>0.
\]

{\bf Case 2: $0<a\le 1$.} In this case, it is clear that $\widetilde{F}_c'>0$ on $(0,\infty)$ and, hence, one has
\begin{equation}\label{eq-deltan}
 \Delta_n=\frac{f(x_n)\log(1+\ell_n)}{f(x_n+\ell_n-1)}=\left(\frac{\log(1+x_n)}
 {\log(\ell_n+x_n)}\right)^a\log(1+\ell_n).
\end{equation}
Observe that
\[
 \log(1+x_n\vee \ell_n)\le\log(\ell_n+x_n)\le 2\log(1+x_n\vee\ell_n)
\]
and
\begin{equation}\label{eq-xnln}
 [\log(1+x_n)][\log(1+\ell_n)]=[\log(1+x_n\vee \ell_n)][\log(1+x_n\wedge \ell_n)].
\end{equation}
This implies $\Lambda_n/2\le\Delta_n\le\Lambda_n$, where
\begin{equation}\label{eq-Ln}
 \Lambda_n=[\log(1+x_n\wedge\ell_n)]^a[\log(1+\ell_n)]^{1-a}.
\end{equation}
By (\ref{eq-2statel2cut2}), when $a=1$, $\mathcal{F}^\mathcal{P}$ has a $L^2$-cutoff if and only if $x_n\wedge\ell_n\ra\infty$. When $0<a<1$, $\mathcal{F}^\mathcal{P}$ has a $L^2$-cutoff if and only if $\ell_n\ra\infty$. For $a=1$, assuming $x_n\wedge\ell_n\ra\infty$ yields
$\Delta_n=(\log x_n)\wedge(\log\ell_n)+O(1)$ and, by (\ref{eq-2statel2cut3}),
\[
 T_{n,2}(\mathbf{0},\epsilon)=\frac{(\log x_n)\wedge(\log \ell_n)}{2r_np_{n,1}}+O\left(\frac{\sqrt{(\log x_n)\wedge(\log \ell_n)}}{r_np_{n,1}}\right),\quad\forall \epsilon>0.
\]
For $0<a<1$, suppose $\ell_n\ra\infty$. Note that
\[
 1+x_n\vee\ell_n\le x_n+\ell_n\le 2(1+x_n\vee\ell_n).
\]
This implies $\log(x_n+\ell_n)=\log(1+x_n\vee\ell_n)+O(1)$ and, by (\ref{eq-xnln}) and (\ref{eq-Ln}),
\begin{align}
 \Delta_n&=\Lambda_n\left(1+O\left(\frac{1}{\log(x_n\vee\ell_n)}\right)\right)\notag\\
 &=[\log(1+x_n\wedge\ell_n)]^a(\log \ell_n)^{1-a}\left(1+O\left(\frac{1}{\log(x_n\vee\ell_n)}
 +\frac{1}{\ell_n\log\ell_n}\right)\right)\notag\\
 &=[\log(1+x_n\wedge\ell_n)]^a(\log \ell_n)^{1-a}+O(1).\notag
\end{align}
By (\ref{eq-2statel2cut3}), we receive
\[
 T_{n,2}(\mathbf{0},\epsilon)=\frac{\zeta_n}{2r_np_{n,1}}+O\left(\frac{\sqrt{\zeta_n}}
 {r_np_{n,1}}\right),\quad\forall\epsilon>0.
\]
where $\zeta_n=[\log(1+x_n\wedge\ell_n)]^a(\log \ell_n)^{1-a}$.

When a cutoff fails to exist, the bound on the mixing time follows is given by (\ref{eq-nocutmixing}) and the details is omitted.
\end{proof}

\begin{proof}[Proof of Theorem \ref{t-machinery}]
The proof of Theorem \ref{t-machinery} follows immediately from Theorem \ref{t-prod2state2} with the replacement of
\[
 A_{n,i}=A_{x_n+i-1},\quad B_{n,i}=B_{x_n+i-1},\quad p_{n,i}=\frac{p_{x_n+i-1}}{q_n}.
\]
The cutoff times in (\ref{eq-Gcutofftime}) and (\ref{eq-cutofftime}) are somewhat different up to a multiple constant $q_n$ and this result in the accelerating constant $q_n$ in $\mathcal{G}$.
\end{proof}

The goal of the following example is to remark some optimality of Theorem \ref{t-main} and we shall show in the following that, for some $c>0$, the limits in conditions (2)-(3) are not sufficient for an $L^2$-cutoff.

\begin{ex}\label{ex-counterexample}
Consider the triangular arrays $\mathcal{F},\mathcal{P}$ in (\ref{eq-f})-(\ref{eq-p}) with
\[
 \ell_n=2n,\quad\mathcal{S}_{n,i}=\{0,1\},\quad L_{n,i}=\left(\begin{array}{cc}-A_{n,i}&A_{n,i}\\
 B_{n,i}&-B_{n,i}\end{array}\right)
\]
and
\[
 A_{n,i}=\begin{cases}1/n&\forall 1\le i\le n,\\1/\sqrt{n}&\forall n<i\le 2n,\end{cases}\quad B_{n,i}=1,\quad p_{n,i}=\begin{cases}i/n^3&\forall 1\le i\le n,\\(\log i)/n^2&\forall n<i\le 2n.\end{cases}
\]
We first prove that $\mathcal{F}^\mathcal{P}=(\mu_n,\mathcal{S}_n,L_n,\pi_n)_{n=1}^\infty$ has no $L^2$-cutoff. For $n\ge 1$ and $1\le i\le 2n$, set $\rho_{n,i}=p_{n,i}(A_{n,i}+B_{n,i})$ and
\[
 D_n(t)=\sum_{i=1}^{2n}A_{n,i}e^{-2\rho_{n,i}t},\quad
 \mathcal{T}_n(\epsilon)=\min\{t\ge 0|D_n(t)\le\epsilon\}.
\]
By Proposition \ref{p-prod}, $\mathcal{F}^\mathcal{P}$ has a $L^2$-cutoff if and only if $\mathcal{T}_n(\epsilon)\sim\mathcal{T}_n(\delta)$ for all $\epsilon,\delta\in(0,\infty)$. Note that, for $A>0$,
\[
 \sum_{i=1}^nA_{n,i}e^{-2\rho_{n,i}An^2}=\frac{1}{n}
 \sum_{i=1}^n\exp\left\{-\frac{2Ai(1+1/n)}{n}\right\}\sim\int_0^1e^{-2As}ds
 =\frac{1-e^{-2A}}{2A}
\]
and
\begin{align}
 \sum_{i=n+1}^{2n}A_{n,i}e^{-2\rho_{n,i}An^2}&=\frac{1}{\sqrt{n}}\sum_{i=n+1}^{2n}
 \exp\left\{-2A(\log i)\left(1+\frac{1}{\sqrt{n}}\right)\right\}\notag\\
 &\sim\frac{1}{\sqrt{n}}\sum_{i=n+1}^{2n}e^{-2A\log i}=\frac{1}{\sqrt{n}}\sum_{i=n+1}^{2n}i^{-2A}.\notag
\end{align}
It is an easy exercise to show that
\[
 0<\int_n^{2n}s^{-2A}ds-\sum_{i=n+1}^{2n}i^{-2A}\le n^{-2A}
\]
and
\[
 \int_n^{2n}s^{-2A}ds=\begin{cases}\log 2&\text{for }A=1/2,\\\frac{2^{1-2A}-1}{1-2A}n^{1-2A}&\text{for }A\ne1/2.\end{cases}
\]
As a consequence of the above computations, one has
\[
 \lim_{n\ra\infty}D_n(An^2)=\begin{cases}\infty&\text{for }0<A<1/4,\\
 2(\sqrt{2}-e^{-1/2})&\text{for }A=1/4,\\(1-e^{-2A})/(2A)&\text{for }A>1/4,\end{cases}
\]
and this leads to
\begin{equation}\label{eq-tnce}
 \mathcal{T}_n(\epsilon)\sim\begin{cases}n^2/4&\text{for }\epsilon\in(2(1-e^{-1/2}),\infty),\\ C_\epsilon n^2&\text{for }\epsilon\in(0,2(1-e^{-1/2})],\end{cases}
\end{equation}
where $C_\epsilon\ge 1/4$ is the constant such that $(1-e^{-2C_\epsilon})/(2C_\epsilon)=\epsilon$. As the mapping $s\mapsto (1-e^{-s})/s$ is strictly decreasing on $(0,\infty)$, $C_\epsilon>C_\delta$ for $\delta>\epsilon\ge 2(1-e^{-1/2})$. This proves that $\mathcal{F}^\mathcal{P}$ has no $L^2$-cutoff.

Next, we compute the $L^2$-mixing time. By Proposition \ref{p-prod}, (\ref{eq-tnce}) leads to $T_{n,2}(\mathbf{0},\epsilon)\asymp n^2$ for all $\epsilon>0$. Further, by applying the fact of $\alpha(c)\ge\sqrt{\log(1+c)}$ to (\ref{eq-bmcts1}) with $A=1$, one has
\[
 \frac{\sqrt{\log(1+c)}}{\sqrt{\log(1+c)}+1}T_{n,2}\left(\mathbf{0},\sqrt{c+1}\right)
 \le\tau_n(c)\le T_{n,2}\left(\mathbf{0},\sqrt{c/(1+c)}\right),
\]
for all $0<c<(1+1/n)^n(1+1/\sqrt{n})^n-1$. As a result, the constant in (\ref{eq-taun}) satisfies $\tau_n(c)\asymp n^2$ for all $c>0$.

Now, we examine the limits in Theorem \ref{t-main}. Let $j_n(c)$ be the constant in (\ref{eq-jn}), $\widetilde{j}_n(c)$ be the constant in (\ref{eq-jnprod}) and set
\[
 \{\varrho_{n,l}|0\le l<2^{2n}\}=\sigma(-L_n),
\]
where $\rho_{n,l}\le\rho_{n,l+1}$ and $\varrho_l\le\varrho_{l+1}$. By Lemma \ref{l-jtaucomp}, one has $\rho_{n,\widetilde{j}_n(\log(1+c))}\le\varrho_{n,j_n(c)}\le\rho_{n,\widetilde{j}_n(c)}$. It is easy to show that
\[
 \widetilde{j}_n(c)=\begin{cases}cn(1+o(1))&\forall 0<c<1,\\n+(c-1)\sqrt{n}(1+o(1))&\forall c>1,\end{cases}
\]
which implies
\[
 \rho_{n,\widetilde{j}_n(c)}\sim\begin{cases}cn^{-2}&\forall 0<c<1,\\(\log n)n^{-2}&\forall c>1.\end{cases}
\]
Consequently, we obtain
\[
 \varrho_{n,j_n(c)}\begin{cases}\asymp n^{-2}&\forall 0<c<1,\\\sim(\log n)n^{-2}&\forall c>e-1,\end{cases}
\]
and this leads to
\[
 T_{n,2}(\mathbf{0},\epsilon)\varrho_{n,j_n(c)}\begin{cases}\asymp 1&\forall 0<c<1,\\\ra\infty&\forall c>e-1,\end{cases}\quad
 \tau_n(c)\varrho_{n,j_n(c)}\begin{cases}\asymp 1&\forall 0<c<1,\\\ra\infty&\forall c>e-1,\end{cases}
\]
for all $\epsilon>0$.
\end{ex}


\appendix

\section{Proof of Lemma \ref{l-ltcomp}}

We first prove (1). Note that $\lambda_V(c)\in(0,\infty)$ for $c\in(0,\mathcal{L}_V(0))$. By Lemma \ref{l-tvc}, there is $\gamma\ge \lambda_V(c)$ such that $e^{\tau_V(c)\gamma}=1+V(\gamma)$. This implies
\begin{align}
 \mathcal{L}_V(\tau_V(c))&\ge \int_{(0,\gamma]}e^{-\tau_V(c)\lambda}dV(\lambda)\ge
 \frac{V(\gamma)}{e^{\tau_V(c)\gamma}}
 =\frac{V(\gamma)}{1+V(\gamma)}\notag\\&\ge\frac{V(\lambda_V(c))}
 {1+V(\lambda_V(c))}\ge\frac{c}{1+c},\notag
\end{align}
where the last two inequalities use the monotonicity of $x\mapsto x/(1+x)$ on $(0,\infty)$.

Next, we consider the second inequality of (1). By Lemma \ref{l-ibp}, one has
\[
 \mathcal{L}_V(t)=t\int_{(0,\infty)}V(\lambda)e^{-t\lambda}d\lambda\le c+t\int_{[\lambda_V(c),\infty)}V(\lambda)e^{-t\lambda}d\lambda.
\]
Note that $V(\lambda)\le e^{\tau_V(c)\lambda}-1\le e^{\tau_V(c)\lambda}$ for $\lambda\ge\lambda_V(c)$. This implies, for $t>\tau_V(c)$,
\[
 \int_{[\lambda_V(c),\infty)}V(\lambda)e^{-t\lambda}d\lambda
 \le \int_{[\lambda_V(c),\infty)}e^{-(t-\tau_V(c))\lambda}d\lambda
 =\frac{e^{-(t-\tau_V(c))\lambda_V(c)}}{t-\tau_V(c)}.
\]
The desired inequality is then given by the replacement of $t$ with $\tau_V(c)+s$.

For (2), the first identity is obvious from the continuity of $\mathcal{L}_V$. For the second inequality, one may use Lemma \ref{l-ibp} to write that, for $r\ge 0$ and $s>0$,
\[
 \mathcal{L}_V(T_V(\epsilon)+r+s)=(T_V(\epsilon)+r+s)\int_{(0,\infty)}V(\lambda)
 e^{-(T_V(\epsilon)+r+s)\lambda}d\lambda.
\]
Note that
\[
 (T_V(\epsilon)+r+s)\int_{(0,\lambda_V(c_1))}V(\lambda)
 e^{-(T_V(\epsilon)+r+s)\lambda}d\lambda\le c_1,
\]
and
\[
 (T_V(\epsilon)+r+s)\int_{[\lambda_V(c_1),\lambda_V(c_2))}V(\lambda)
 e^{-(T_V(\epsilon)+r+s)\lambda}d\lambda\le c_2e^{-(T_V(\epsilon)+r+s)\lambda_V(c_1)},
\]
and
\begin{align}
 \int_{[\lambda_V(c_2),\infty)}V(\lambda)&e^{-(T_V(\epsilon)+r+s)\lambda}d\lambda\notag\\
 \le& e^{-r\lambda_V(c_2)}\int_{[\lambda_V(c_2),\infty)}V(\lambda)
 e^{-(T_V(\epsilon)+s)\lambda}d\lambda.\notag
\end{align}
The desired inequality is then given by adding up the above three bounds and applying the observation of
\begin{align}
 \int_{[\lambda_V(c_2),\infty)}V(\lambda)e^{-(T_V(\epsilon)+s)\lambda}d\lambda
 &\le\int_{(0,\infty)}V(\lambda)e^{-(T_V(\epsilon)+s)\lambda}d\lambda\notag\\
 &=\frac{\mathcal{L}_V(T_V(\epsilon)+s)}{T_V(\epsilon)+s}
 \le\frac{\epsilon}{T_V(\epsilon)+s},\notag
\end{align}
where the second-to-last equality applies Lemma \ref{l-ibp} again.

\section{Techniques for product chains}

Let $(\mu_i,\mathcal{S}_i,L_i,\pi_i)_{i=1}^n$ be irreducible and reversible continuous time finite Markov chains, $(p_i)_{i=1}^n$ be positive constants satisfying $\sum_{i=1}^np_i\le 1$, and $(\mu,\mathcal{S},L,\pi)$ be a continuous time Markov chain with $\mathcal{S}=\mathcal{S}_1\times\cdots\times\mathcal{S}_n$, $\mu=\mu_1\times\cdots\times\mu_n$, $\pi=\pi_1\times\cdots\times\pi_n$ and
\[
 L=\sum_{i=1}^np_iI_1\otimes\cdots\otimes I_{i-1}\otimes L_i\otimes I_{i+1}\otimes\cdots\otimes I_n,
\]
where $I_i$ is the identity matrix indexed by $\mathcal{S}_i$. Let $\lambda_{i,0}=0,\lambda_{i,1},...,\lambda_{i,|\mathcal{S}_i|-1}$ be eigenvalues of $-L_i$ with $L^2(\pi_i)$-orthonormal right eigenvectors $\phi_{i,0}=\mathbf{1},\phi_{i,1},...,\phi_{i,|\mathcal{S}_i|-1}$. Set $\Gamma=\{j=(j_1,...,j_n)|0\le j_i<|\mathcal{S}_i|,\,\forall 1\le i\le n\}$ and, for $J=(j_1,...,j_n)\in\Gamma$, define $\lambda_J=\sum_{i=1}^np_i\lambda_{i,j_i}$ and $\phi_J=\prod_{i=1}^n\phi_{i,j_i}$. It is easy to see that, for $J\in\Gamma$, $\lambda_J$ is an eigenvalue of $-L$ with $L^2(\pi)$-orthonormal right eigenvector $\phi_J$. Write
\begin{equation}\label{eq-varrho}
 \left\{\varrho_l\bigg|1\le l<\prod_{i=1}^n|\mathcal{S}_i|\right\}=\{\lambda_J|J\in\Gamma,J\ne\mathbf{0}\}
\end{equation}
and
\begin{equation}\label{eq-rho}
 \left\{\rho_l\bigg|1\le l\le \sum_{i=1}^n|\mathcal{S}_i|-n\right\}=\{p_i\lambda_{i,j}|1\le j<|\mathcal{S}_i|,\,1\le i\le n\}
\end{equation}
in the way that $\rho_l\le\rho_{l+1}$ and $\varrho_l\le\varrho_{l+1}$. We rearrange $\mu(\phi_J)$'s and $\mu_i(\phi_{i,j})$'s accordingly and write them as $\psi_l$'s and $\varphi_l$'s. Consider the following setting. For $c>0$, set
\begin{equation}\label{eq-jj}
 j(c)=\min\left\{j\ge 1\bigg|\sum_{i=1}^j\psi_j^2>c\right\},\quad
 \widetilde{j}(c)=\min\left\{j\ge 1\bigg|\sum_{i=1}^j\varphi_j^2>c\right\}.
\end{equation}

\begin{lem}\label{l-jtaucomp}
Referring to the setting in \textnormal{(\ref{eq-varrho})}, \textnormal{(\ref{eq-rho})} and \textnormal{(\ref{eq-jj})}, one has
\[
 \varrho_{j(c)}\le\rho_{\widetilde{j}(c)}\le\varrho_{j(e^c-1)},\quad\forall c>0,
\]
where $\min\emptyset:=\infty$.
\end{lem}
\begin{proof}
Suppose that $\lambda_{i,j}\le \lambda_{i,j+1}$. Fix $c>0$ and let $J=(J_1,J_2,...,J_n)\in\Gamma$ be a vector such that
\begin{equation}\label{eq-rhol}
 \{\rho_l|1\le l\le \widetilde{j}(c)\}=\bigcup_{i=1}^n\{p_i\lambda_{n,j}|1\le j\le J_i\}.
\end{equation}
Note that $\{\lambda_{j_ie_i}|1\le j_i\le J_i,\,1\le i\le n\}=\{\rho_l|1\le l\le \widetilde{j}(c)\}$, where $e_i$ is a vector with $1$ in the $i$th coordinate and $0$ in the others. This implies that there is an integer $N\ge 1$ such that $\varrho_N=\rho_{\widetilde{j}(c)}$ and
\[
 \{\varphi_l|1\le l\le \widetilde{j}(c)\}\subset\{\psi_l|1\le l\le N\}.
\]
Clearly, one has
\[
 \sum_{l=1}^N\psi_l^2\ge \sum_{l=1}^{\widetilde{j}(c)}\varphi_l^2>c.
\]
As a consequence, this leads to $j(c)\le N$ and then $\varrho_{j(c)}\le\varrho_N=\rho_{\widetilde{j}(c)}$, which proves the first inequality.

For the second inequality, let $J$ be the vector as before. Up to a permutation of $\{\mathcal{S}_i|1\le i\le n\}$, we may assume $J_1\ge 1$ and $p_1\lambda_{1,J_1}=\rho_{\widetilde{j}(c)}$. Set $J'=(J_1',...,I_n')$, where $J_1'=J_1-1$ and $J_i'=J_i$ for $2\le i\le n$. For $\mathcal{I}=(i_1,...,i_n)$ and $\mathcal{J}=(j_1,...,j_n)$, we write $\mathcal{I}\preceq \mathcal{J}$ if $i_k\le j_k$ for all $1\le k\le n$. Using the fact of $\log (1+t)\le t$, one may derive
\begin{align}
 \sum_{\mathcal{J}:\mathbf{0}\preceq\mathcal{J}\preceq J'}\phi_{\mathcal{J}}^2&=\sum_{\mathcal{J}:\mathcal{J}\preceq J'}\phi_{\mathcal{J}}^2-1
 =\prod_{i=1}^n\sum_{j=0}^{J_i'}|\mu_i(\phi_{i,j})|^2-1\notag\\
 &\le\exp\left\{\sum_{i=1}^n\sum_{j=1}^{J_i'}|\mu_i(\phi_{i,j})|^2\right\}-1\le e^c-1.\notag
\end{align}
Further, by the setting in (\ref{eq-rhol}), it is easy to see that $\lambda_{je_i}\ge \rho_{\widetilde{j}(c)}$ for all $j>J_i'$ and $1\le i\le n$. If $\mathcal{I}=(i_1,...,i_n)\npreceq J'$, then there is $1\le k\le n$ such that $i_k>J_k'$ and this implies $\lambda_{\mathcal{I}}\ge\lambda_{i_ke_k}\ge\rho_{\widetilde{j}(c)}$. Consequently, we have $\varrho_{j(e^c-1)}\ge\rho_{\widetilde{j}(c)}$, as desired.
\end{proof}

\end{document}